\documentclass{amsart}

\usepackage{tikz}\usepackage{tikz-cd} %diagrams

\tikzset{curve/.style={settings={#1},to path={(\tikztostart)
    .. controls ($(\tikztostart)!\pv{pos}!(\tikztotarget)!\pv{height}!270:(\tikztotarget)$)
    and ($(\tikztostart)!1-\pv{pos}!(\tikztotarget)!\pv{height}!270:(\tikztotarget)$)
    .. (\tikztotarget)\tikztonodes}},
    settings/.code={\tikzset{quiver/.cd,#1}
        \def\pv##1{\pgfkeysvalueof{/tikz/quiver/##1}}},
    quiver/.cd,pos/.initial=0.35,height/.initial=0}

\tikzset{tail reversed/.code={\pgfsetarrowsstart{tikzcd to}}}
\tikzset{2tail/.code={\pgfsetarrowsstart{Implies[reversed]}}}
\tikzset{2tail reversed/.code={\pgfsetarrowsstart{Implies}}}
\tikzset{no body/.style={/tikz/dash pattern=on 0 off 1mm}}

\usepackage{bbold}
\usepackage{amsmath, amsfonts, amssymb} %type maths
\usepackage{stmaryrd}
\usepackage{mathtools}
\usepackage{enumerate} %lists
\usepackage{multicol} %columns
\usepackage{array} %table options
\usepackage{forest} %trees
\usepackage{xcolor}
\usepackage{comment}
\usepackage{multirow}

\newcommand{\C}{\mathbb{C}}
\newcommand{\D}{\mathbb{D}}

\DeclareMathOperator{\ob}{ob}
\DeclareMathOperator{\Cat}{Cat}
\DeclareMathOperator{\CAT}{CAT}
\DeclareMathOperator{\Psh}{Psh}
\DeclareMathOperator{\Set}{Set}
\DeclareMathOperator{\Lan}{Lan}

\newtheorem{proposition}{Proposition}[section]
\newtheorem{lemma}[proposition]{Lemma}
\newtheorem{theorem}[proposition]{Theorem}
\newtheorem*{theorem*}{Theorem}

\theoremstyle{definition}
\newtheorem{definition}[proposition]{Definition}
\newtheorem{remark}[proposition]{Remark}
\newtheorem{example}[proposition]{Example}

\usepackage[colorlinks=true]{hyperref}
\hypersetup{allcolors=[rgb]{0.1,0.1,0.4}}

\title[Pseudocommutativity and Lax Idempotency]{Pseudocommutativity and Lax Idempotency\\for Relative Pseudomonads}
\author[A. Slattery]{Andrew Slattery}
\address{School of Mathematics, University of Leeds}
\email{mmawsl@leeds.ac.uk}
\subjclass{Primary 18N15; Secondary 18D65, 18A05, 18M65}
\date{\today}

\begin{document}
\maketitle

\begin{abstract}
    We extend the classical work of Kock on strong and commutative monads, as well as the work of Hyland and Power for 2-monads, in order to define strong and pseudocommutative relative pseudomonads. To achieve this, we work in the more general setting of 2-multicategories rather than monoidal 2-categories. We prove analogous implications to the classical work: that a strong relative pseudomonad is a pseudo-multifunctor, and that a pseudocommutative relative pseudomonad is a multicategorical pseudomonad. Furthermore, we extend the work of López Franco with a proof that a lax-idempotent strong relative pseudomonad is pseudocommutative.

    We apply the results of this paper to the example of the presheaf relative pseudomonad.
\end{abstract}

\section{Introduction}

\noindent\textbf{Context and motivation.} The classical theory of monads provides a framework with which to study algebraic structures on objects of a category. A landmark in this field is Kock's theory of commutative monads \cite{kock1970}, developed in the setting of symmetric monoidal categories. The basic notion in this theory is that of a \emph{strong monad}, which comprises a monad on a symmetric monoidal category equipped with a natural transformation with components
\[t_{X,Y} : X \otimes TY \to T(X \otimes Y),\] called the \emph{strength}. The underlying endofunctor of a strong monad is a lax monoidal functor, and the monad unit is a monoidal natural transformation. Furthermore, Kock showed that the monad is \emph{commutative} (a property of a given strength) if and only if the monad is a monoidal monad, which is to say that the monad multiplication is monoidal.
Some nice properties follow when this happens. For example, if a symmetric monoidal category $\C$ has a closed structure and $T$ is a commutative monad on $\C$, then the closed structure gives rise to one on Eilenberg-Moore category of $T$-algebras.

Two-dimensional monad theory \cite{bkp1989} has traditionally studied the strict notion of a 2-monad, along with their algebras and lax, pseudo-, and strict algebra morphisms. Kelly \cite{kelly74} and Hyland \& Power \cite{hp2002} extended Kock's theory to 2-monads, defining \emph{pseudocommutative 2-monads}. Some aspects of the theory become more subtle; for example, one must distinguish between braiding and symmetry, and between closed structures and pseudo-closed structures. In this setting, an important result is López Franco's theorem \cite{lopezfranco2011} that a lax-idempotent pseudomonad is pseudocommutative (extending work of Power, Cattani and Winskel in \cite{pcw98}).

For some applications, it is useful to consider the notions of a pseudomonad \cite{bunge74,ng21,lack00,marmolejo99}, in which the axioms for a 2-monad hold only up to coherent isomorphisms, and of a relative pseudomonad \cite{fghw}, in which one further abandons the requirement of having an underlying endofunctor. The latter can be seen as a 2-categorical counterpart of the notion of a relative monad \cite{acu2015,arkormcdermott2023,Lobbia2023}.

The aim of this paper is to provide an analogue of the theory of Hyland \& Power and of López Franco for relative pseudomonads. We are motivated to do so by the presheaf construction; here, pseudocommutativity and lax idempotency are particularly intuitive and correspond to important properties of the presheaf construction. As a byproduct of the work in this paper, we obtain a theory of commutativity for relative monads. We also expect a close relationship between the work in this paper and that on strength for pseudomonads in \cite{paquetsaville23}, which have been developed independently.

\noindent\textbf{Main contributions.} We are naturally led to work in a multicategorical setting, as was already partially done by Hyland \& Power in \cite{hp2002}. This step is unavoidable if we wish to avoid dealing with associator and unitor coherences while still having our work apply directly to the 2-categories $\Cat$ and $\CAT$ of small and locally-small categories. Multicategories subsume monoidal categories, with monoidal categories corresponding to the subclass of `representable multicategories', as laid out by Hermida in \cite{hermida2000}. Thus we work in general with $n$-ary maps $f : X_1,...,X_n \to Y$, and our definitions reflect this.

We define the notion of strong relative pseudomonad (Definition~\ref{paramrelpsmonad}), and prove that for  a strong relative pseudomonad, the underlying pseudofunctor becomes a multi-pseudofunctor (Proposition~\ref{param=>psfunctor}) and the unit becomes multicategorical (part of Theorem~\ref{pscom=>multmonad}). We also define the notion of pseudocommutative relative pseudomonad, which in our setting is particularly appealing; it amounts to asking for an isomorphism
\[(f^t)^s \cong (f^s)^t.\] We then prove that every pseudocommutative relative pseudomonad is a multicategorical relative pseudomonad (Theorem~\ref{pscom=>multmonad}). We define the notion of lax idempotency for strong relative pseudomonads, extending earlier definitions in \cite{diliberti2023kzpseudomonads} and \cite{fghw}, and prove that every lax-idempotent strong relative pseudomonad is pseudocommutative (Theorem~\ref{laxid=>pscom}). We apply these definitions and results to the example of presheaves (Theorem~\ref{psh is all the things}).

\noindent\textbf{Roadblocks and technical challenges.} As with any venture at generalisation, we lose some implications and equivalences. For example, while Kock \cite{kock1970} proves an equivalence between strong monads and monads which are lax monoidal as functors, as well as one between commutative and monoidal monads, in our setting we will only have implications in the forward direction. Another assumption we must drop if we are to apply our results to the presheaf construction is that of closure; while $\Cat$ is closed, $\CAT$ is not (again due to size issues—the functor category $[X,Y]$ need not be locally small even when both $X$ and $Y$ are). This means in particular that the proof that lax-idempotent pseudomonads are pseudocommutative given by López Franco in \cite{lopezfranco2011} cannot be readily transported to our setting, as it makes heavy use of closure.

Other trade-offs come from working in the setting of multicategories. The classical strength employs a binary map $X \otimes TY \to T(X \otimes Y)$; we will need an $n$-ary formulation in order to extend to the notion of strength to the multicategorical setting. In general, we are able to obviate associativity and unitor coherences, at the expense of having to work in an unbiased way on general n-ary morphisms, instead of being able to consider only binary and nullary morphisms.

\noindent\textbf{Organisation of the paper.} Section 2 reviews the definition of a relative pseudomonad and some immediate results, and introduces the example of the presheaf relative pseudomonad. Our new work begins in section 3, in which we introduce the setting of 2-multicategories, define a notion of relative pseudomonad suitable for this setting (strong relative pseudomonad) and prove that every strong relative pseudomonad is a pseudo-multifunctor. In section 4 we focus on the class of strong relative pseudomonads which are pseudocommutative, and prove that every pseudocommutative relative pseudomonad is a multicategorical relative pseudomonad. Section 5 discusses a particularly nice class of strong relative pseudomonads, the lax-idempotent strong relative pseudomonads, and proves that every lax-idempotent strong relative pseudomonad is pseudocommutative. We close in Section 6 by applying our results to the case of $\Psh$, the presheaf relative pseudomonad.

\section{Background}

We recall the definition of a relative pseudomonad from \cite{fghw}; for our purposes it will suffice to consider relative pseudomonads along a fixed 2-functor $J : \D \to \C$ between 2-categories $\C$ and $\D$ (as opposed to a pseudofunctor between bicategories).

\begin{definition}
    (Relative pseudomonad) Let $\C,\D$ be 2-categories and let $J : \D \to \C$ be a 2-functor. A \emph{relative pseudomonad} $(T,i,{}^*;\eta,\mu,\theta)$ along $J$ comprises
    \begin{itemize}
        \item for $X \in \ob\D$ an object $TX \in \ob\C$ and map $i_X : JX \to TX$ (called a \emph{unit map}), and
        \item for $X,Y \in \ob\D$ a functor \[\C(JX,TY) \xrightarrow{(-)^*} \C(TX,TY)\] (called an \emph{extension functor}).
    \end{itemize}
    The units and extensions furthermore come equipped with three invertible families of 2-cells
    \begin{itemize}
        \item $\eta_f : f \to f^*i_X$ for $f : JX \to TY$,
        \item $\mu_{f,g} : (f^*g)^* \to f^*g^*$ for $g : JX \to TY$, $f : JY \to TZ$, and
        \item $\theta_X : (i_X)^* \to 1_{TX}$ for $X \in \ob\D$,
    \end{itemize} satisfying the following two coherence conditions:

    \begin{enumerate}
        \item[(1)] for every $f : JX \to TY$, $g : JW \to TX$ and $h : JV \to TW$ the diagram
    % https://q.uiver.app/?q=WzAsNSxbMCwwLCIoKGZeKmcpXipoKV4qIl0sWzAsMSwiKGZeKmdeKmgpXioiXSxbMiwwLCIoZl4qZyleKmheKiJdLFsxLDEsImZeKihnXipoKV4qIl0sWzIsMSwiZl4qZ14qaF4qIl0sWzAsMSwiKFxcbXVfe2YsZ31oKV4qIiwyXSxbMCwyLCJcXG11X3tmXipnLGh9Il0sWzEsMywiXFxtdV97ZixnXipofSIsMl0sWzMsNCwiZl4qXFxtdV97ZyxofSIsMl0sWzIsNCwiXFxtdV97ZixnfWheKiJdXQ==
\[\begin{tikzcd}[ampersand replacement=\&]
	{((f^*g)^*h)^*} \&\& {(f^*g)^*h^*} \\
	{(f^*g^*h)^*} \& {f^*(g^*h)^*} \& {f^*g^*h^*}
	\arrow["{(\mu_{f,g}h)^*}"', from=1-1, to=2-1]
	\arrow["{\mu_{f^*g,h}}", from=1-1, to=1-3]
	\arrow["{\mu_{f,g^*h}}"', from=2-1, to=2-2]
	\arrow["{f^*\mu_{g,h}}"', from=2-2, to=2-3]
	\arrow["{\mu_{f,g}h^*}", from=1-3, to=2-3]
\end{tikzcd}\] commutes (the \emph{associativity axiom}), and
        \item[(2)] for every $f : JX \to TY$ the diagram
    % https://q.uiver.app/?q=WzAsNCxbMCwwLCJmXioiXSxbMSwwLCIoZl4qaSleKiJdLFsyLDAsImZeKmleKiJdLFsyLDEsImZeKjEiXSxbMCwzLCIiLDIseyJsZXZlbCI6Miwic3R5bGUiOnsiaGVhZCI6eyJuYW1lIjoibm9uZSJ9fX1dLFswLDEsIihcXGV0YV9mKV4qIl0sWzEsMiwiXFxtdV97ZixpfSJdLFsyLDMsImZeKlxcdGhldGEiXV0=
\[\begin{tikzcd}[ampersand replacement=\&]
	{f^*} \& {(f^*i)^*} \& {f^*i^*} \\
	\&\& {f^*1}
	\arrow[Rightarrow, no head, from=1-1, to=2-3]
	\arrow["{(\eta_f)^*}", from=1-1, to=1-2]
	\arrow["{\mu_{f,i}}", from=1-2, to=1-3]
	\arrow["{f^*\theta}", from=1-3, to=2-3]
\end{tikzcd}\] commutes (the \emph{unit axiom}).
    \end{enumerate}
\end{definition}
We usually omit subscripts from the unit maps $i : JX \to TX$; we will also refer to a given relative pseudomonad $(T,i,{}^*;\eta,\mu,\theta)$ simply as $(T,i,{}^*)$ or $T$, with the rest of the structure inferred.

Given a relative pseudomonad $T$ along $J : \D \to \C$, the function $ \ob\D \to \ob C : X \mapsto TX$ can be given the structure of a pseudofunctor, with functors between hom-categories given by \[\D(X,Y) \to \C(TX,TY) : f \mapsto (i_Y \circ Jf)^*.\]

\begin{remark}
    A relative pseudomonad along the identity $1 : \C \to \C$ induces and is induced by an ordinary pseudomonad with the same action on objects (see \cite{fghw} Remark~4.5).
\end{remark}

We can infer more equalities between a relative pseudomonad's structural 2-cells. The following lemma is from \cite{fghw}; the proof is analogous to the proof that three of the original five axioms for a monoidal category are redundant \cite{kelly1964}, which also has a version for (ordinary) pseudomonads \cite{marm1997}.

\begin{lemma}\label{3moreeqs}
    Let $T$ be a relative pseudomonad along $J : \D \to \C$. Then in addition to the two equalities of 2-cells given by definition, the following three diagrams also commute:
    \begin{enumerate}
        \item[(1)] for every $f : JX \to TY$ and $g : JW \to TX$, the diagram
        % https://q.uiver.app/?q=WzAsMyxbMCwwLCJmXipnIl0sWzEsMCwiKGZeKmcpXippIl0sWzEsMSwiZl4qZ14qaSJdLFswLDEsIlxcZXRhX3tmXipnfSJdLFsxLDIsIlxcbXVfe2YsZ31pIl0sWzAsMiwiZl4qXFxldGFfZyIsMl1d
\[\begin{tikzcd}[ampersand replacement=\&]
	{f^*g} \& {(f^*g)^*i} \\
	\& {f^*g^*i}
	\arrow["{\eta_{f^*g}}", from=1-1, to=1-2]
	\arrow["{\mu_{f,g}i}", from=1-2, to=2-2]
	\arrow["{f^*\eta_g}"', from=1-1, to=2-2]
\end{tikzcd}\] commutes.
        \item[(2)] for every $f : JX \to TY$, the diagram
        % https://q.uiver.app/?q=WzAsMyxbMCwwLCIoaV4qZileKiJdLFsxLDEsImZeKiJdLFsxLDAsImleKmZeKiJdLFsyLDEsIlxcdGhldGEgZl4qIl0sWzAsMSwiKFxcdGhldGEgZileKiIsMl0sWzAsMiwiXFxtdV97aSxmfSJdXQ==
\[\begin{tikzcd}[ampersand replacement=\&]
	{(i^*f)^*} \& {i^*f^*} \\
	\& {f^*}
	\arrow["{\theta f^*}", from=1-2, to=2-2]
	\arrow["{(\theta f)^*}"', from=1-1, to=2-2]
	\arrow["{\mu_{i,f}}", from=1-1, to=1-2]
\end{tikzcd}\] commutes, and
        \item[(3)] for every object $X \in \ob\D$, the diagram
        % https://q.uiver.app/?q=WzAsMyxbMCwwLCJpIl0sWzEsMSwiaSJdLFsxLDAsImleKmkiXSxbMCwyLCJcXGV0YV9pIl0sWzIsMSwiXFx0aGV0YSBpIl0sWzAsMSwiIiwyLHsibGV2ZWwiOjIsInN0eWxlIjp7ImhlYWQiOnsibmFtZSI6Im5vbmUifX19XV0=
\[\begin{tikzcd}[ampersand replacement=\&]
	i \& {i^*i} \\
	\& i
	\arrow["{\eta_i}", from=1-1, to=1-2]
	\arrow["{\theta i}", from=1-2, to=2-2]
	\arrow[Rightarrow, no head, from=1-1, to=2-2]
\end{tikzcd}\] commutes.
    \end{enumerate}
\end{lemma}

\begin{example}\label{psh ex}
The example of a relative pseudomonad which will be the focus of this paper is that of the presheaf construction.\[X \mapsto \Psh X := [X^{op},\Set]\]
Write $\Cat$ for the 2-category of small categories, functors and natural transformations, and write $\CAT$ for the 2-category of locally-small categories. Since the category of presheaves on a small category is in general only locally small, it is natural to ask whether $\Psh$ can be given the structure of a relative pseudomonad along the inclusion 2-functor $J : \Cat \to \CAT$.

This is shown in \cite{fghw} via the construction of a relative pseudoadjunction; the structure of a relative pseudomonad is given to $\Psh$ as follows:

\begin{itemize}
    \item for an object $X \in \Cat$ we have $\Psh X \in \CAT$ and unit map $y_X : X \to \Psh X$ given by the Yoneda embedding,
    \item for $X,Y \in \Cat$ and a functor $f : X \to \Psh Y$, the extension $f^* : \Psh X \to \Psh Y$ is given by the left Kan extension of $f$ along the Yoneda embedding
    % https://q.uiver.app/?q=WzAsMyxbMCwwLCJYIl0sWzEsMCwiXFxQc2ggWCJdLFsxLDEsIlxcUHNoIFkiXSxbMCwxLCJ5Il0sWzEsMiwiZl4qIDo9IFxcTGFuX3kgZiJdLFswLDIsImYiLDJdLFs1LDQsIlxcZXRhX2YiLDAseyJzaG9ydGVuIjp7InNvdXJjZSI6MjAsInRhcmdldCI6MjB9fV1d
\[\begin{tikzcd}[ampersand replacement=\&]
	X \& {\Psh X} \\
	\& {\Psh Y}
	\arrow["y", from=1-1, to=1-2]
	\arrow[""{name=0, anchor=center, inner sep=0}, "{f^* := \Lan_y f}", from=1-2, to=2-2]
	\arrow[""{name=1, anchor=center, inner sep=0}, "f"', from=1-1, to=2-2]
	\arrow["{\eta_f}", shorten <=4pt, shorten >=4pt, Rightarrow, from=1, to=0]
\end{tikzcd}\] which also defines the 2-cells $\eta_f : f \to f^*y$ (note that since the Yoneda embedding is fully faithful the maps $\eta_f$ are invertible, as required),
    \item for $f : JX \to TY$ and $g : JW \to TX$, the 2-cell $\mu_{f,g} : (f^*g)^* \to f^*g^*$ is uniquely determined by the universal property of the left Kan extension:
    % https://q.uiver.app/?q=WzAsOCxbMSwxLCJcXFBzaCBYIl0sWzEsMiwiXFxQc2ggWSJdLFswLDAsIlciXSxbMSwwLCJcXFBzaCBXIl0sWzIsMCwiVyJdLFs0LDIsIlxcUHNoIFkiXSxbNCwwLCJcXFBzaCBXIl0sWzUsMSwiXFxQc2ggWCJdLFswLDEsImZeKiJdLFsyLDAsImciLDJdLFsyLDMsInkiXSxbMywwLCJnXioiXSxbNCw1LCJmXipnIiwyXSxbNCw2LCJ5Il0sWzYsNSwiKGZeKmcpXioiLDFdLFs2LDcsImdeKiJdLFs3LDUsImZeKiJdLFs5LDExLCJcXGV0YV9nIiwwLHsic2hvcnRlbiI6eyJzb3VyY2UiOjIwLCJ0YXJnZXQiOjIwfX1dLFsxMiwxNCwiXFxldGFfe2ZeKmd9IiwwLHsic2hvcnRlbiI6eyJzb3VyY2UiOjIwLCJ0YXJnZXQiOjQwfX1dLFsxNCw3LCJcXG11X3tmLGd9IiwwLHsic2hvcnRlbiI6eyJzb3VyY2UiOjQwfX1dXQ==
\[\begin{tikzcd}[ampersand replacement=\&]
	W \& {\Psh W} \& W \&\& {\Psh W} \\
	\& {\Psh X} \&\&\&\& {\Psh X} \\
	\& {\Psh Y} \&\&\& {\Psh Y}
	\arrow["{f^*}", from=2-2, to=3-2]
	\arrow[""{name=0, anchor=center, inner sep=0}, "g"', from=1-1, to=2-2]
	\arrow["y", from=1-1, to=1-2]
	\arrow[""{name=1, anchor=center, inner sep=0}, "{g^*}", from=1-2, to=2-2]
	\arrow[""{name=2, anchor=center, inner sep=0}, "{f^*g}"', from=1-3, to=3-5]
	\arrow["y", from=1-3, to=1-5]
	\arrow[""{name=3, anchor=center, inner sep=0}, "{(f^*g)^*}"{description}, from=1-5, to=3-5]
	\arrow["{g^*}", from=1-5, to=2-6]
	\arrow["{f^*}", from=2-6, to=3-5]
	\arrow["{\eta_g}", shorten <=4pt, shorten >=4pt, Rightarrow, from=0, to=1]
	\arrow["{\eta_{f^*g}}", shorten <=7pt, shorten >=14pt, Rightarrow, from=2, to=3]
	\arrow["{\mu_{f,g}}", shorten <=11pt, Rightarrow, from=3, to=2-6]
\end{tikzcd}\]
    \item for $X \in \Cat$, the 2-cell $\theta_X : y_X^* \to 1$ is also uniquely determined by the universal property of the left Kan extension:
    % https://q.uiver.app/?q=WzAsNixbMSwxLCJcXFBzaCBYIl0sWzAsMCwiWCJdLFsxLDAsIlxcUHNoIFgiXSxbMiwwLCJYIl0sWzMsMSwiXFxQc2ggWCJdLFszLDAsIlxcUHNoIFgiXSxbMSwwLCJ5IiwyXSxbMSwyLCJ5Il0sWzIsMCwiMSJdLFszLDQsInkiLDJdLFszLDUsInkiXSxbNSw0LCJpXioiLDFdLFs1LDQsIjEiLDAseyJjdXJ2ZSI6LTV9XSxbNiw4LCIxIiwwLHsic2hvcnRlbiI6eyJzb3VyY2UiOjIwLCJ0YXJnZXQiOjIwfX1dLFs5LDExLCJcXGV0YV95IiwwLHsic2hvcnRlbiI6eyJzb3VyY2UiOjIwLCJ0YXJnZXQiOjIwfX1dLFsxMSwxMiwiXFx0aGV0YV9YIiwwLHsic2hvcnRlbiI6eyJzb3VyY2UiOjIwLCJ0YXJnZXQiOjIwfX1dXQ==
\[\begin{tikzcd}[ampersand replacement=\&]
	X \& {\Psh X} \& X \& {\Psh X} \\
	\& {\Psh X} \&\& {\Psh X}
	\arrow[""{name=0, anchor=center, inner sep=0}, "y"', from=1-1, to=2-2]
	\arrow["y", from=1-1, to=1-2]
	\arrow[""{name=1, anchor=center, inner sep=0}, "1", from=1-2, to=2-2]
	\arrow[""{name=2, anchor=center, inner sep=0}, "y"', from=1-3, to=2-4]
	\arrow["y", from=1-3, to=1-4]
	\arrow[""{name=3, anchor=center, inner sep=0}, "{i^*}"{description}, from=1-4, to=2-4]
	\arrow[""{name=4, anchor=center, inner sep=0}, "1", curve={height=-30pt}, from=1-4, to=2-4]
	\arrow["1", shorten <=4pt, shorten >=4pt, Rightarrow, from=0, to=1]
	\arrow["{\eta_y}", shorten <=4pt, shorten >=4pt, Rightarrow, from=2, to=3]
	\arrow["{\theta_X}", shorten <=6pt, shorten >=6pt, Rightarrow, from=3, to=4]
\end{tikzcd}\]
\end{itemize}

\end{example}

\section{Strong relative pseudomonads}

The 2-categories $\Cat$ and $\CAT$ possess more structure than simply being 2-categories; they are in particular cartesian monoidal 2-categories. Thus we will seek to develop the Kock's theory of monads on symmetric monoidal closed categories \cite{kock1970} for relative pseudomonads, defining notions of \emph{strong relative pseudomonads} and \emph{pseudocommutative relative pseudomonads}. These will specialise in the one-dimensional ordinary setting to Kock's strong monads and commutative monads, respectively. To avoid some of the coherence isomorphisms inherent to working with monoidal 2-categories, we will work in the related setting of 2-multicategories (see Definition~\ref{2mult}). 

We seek to consider the notion of a relative pseudomonad along $J : \D \to \C$ when $\C$ and $\D$ are 2-multicategories. We will define a `strong relative pseudomonad' from scratch to take this role, and note that a every strong relative pseudomonad induces a canonical relative pseudomonad structure. In order to do this, let us recall the definition of a 2-multicategory \cite{hermida2000} (taking $V = \Cat$ to specialise the $V$-enriched theory).

\begin{definition}\label{2mult}
(2-multicategory) A 2-multicategory $\C$ is a multicategory enriched in $\Cat$. Unwrapping this statement a little, a 2-multicategory $\C$ is given by
    \begin{enumerate}
        \item a collection of objects $X \in \ob\C$, together with
        \item a category of multimorphisms $\C(X_1,...,X_n;Y)$ for all $n \geq 0$ and objects $X_1,...,X_n,Y$ which we call a \emph{hom-category}; an object of the hom-category $\C(X_1,...,X_n;Y)$ is denoted by $f : X_1,...,X_n \to Y$,
        \item an identity multimorphism functor $\mathbf{1}_X : \mathbb{1} \to \C(X;X) : * \mapsto 1_X$ for all $X \in \ob\C$, and
        \item composition functors \[\C(X_1,...,X_n;Y) \times \C(W_{1,1},...,W_{1,m_1};X_1) \times ...\times \C(W_{n,1},...,W_{n,m_n};X_n)\]\begin{align*}
            &\to \C(W_{1,1},...,W_{n,m_n};Y)\\
            (f,g_1,...,g_n) &\mapsto f \circ (g_1,...,g_n)
        \end{align*} for all arities $n,m_1,...,m_n$ and objects $Y,X_1,...,X_n,W_{1,1},...,W_{n,m_n}$ in $\C$.
    \end{enumerate}
    where the identity and composition functors satisfy the usual associativity and identity axioms for an enrichment. 
\end{definition}

As a point of notation, given $f : X_1,...,X_n \to Y$ and $g : W_1,...,W_m \to X_j$ we will abbreviate composites of the form $f \circ (1,...,1,g,1,...,1)$ to $f \circ_j g$ .

\begin{remark} We can relate 2-multicategories to more familiar structures.
\begin{itemize}
    \item Every 2-multicategory $\C$ restricts to a 2-category by considering only the unary hom-categories $\C(X;Y)$.
    \item Monoidal 2-categories (defined in for example \cite{day1997}) have underlying 2-multicategories, where hom-categories $\C(X_1,...,X_n;Y)$ are given by $\C(X_1 \otimes ... \otimes X_n,Y)$ (choosing the leftmost bracketing of the tensor product); this is shown in \cite{hermida2000} Proposition 7.1 (2). For example, both $\Cat$ and $\CAT$ can be given 2-multicategorical structures.
\end{itemize}
\end{remark}

We seek to generalise Kock's notion of a strong monad \cite{kock1970} (and Uustalu's definition of a strong relative monad \cite{uustalu10}) on a monoidal category. A strong monad structure on a monoidal category is given by a map
\[t_{X,Y} : X \otimes TY \to T(X\otimes Y)\]
satisfying some axioms \cite{kock1970}. To define a suitable notion of strong relative pseudomonad in the 2-multicategorical setting, we extend a relative pseudomonad's unary functors $\C(JX,TY) \xrightarrow{(-)^*} \C(TX,TY)$ to general $n$-ary hom-categories
\[\C(B_1,...,JX,...,B_n;TY) \xrightarrow{(-)^{t_i}} \C(B_1,...,TX,...,B_n;TY),\]
which we call \emph{strengthenings}. To use this to construct the map $t$ in the one-dimensional monoidal $J=1$ case, we begin with the unit \[i : X\otimes Y \to T(X \otimes Y).\] Passing to the underlying multicategory, this corresponds to a map \[i : X,Y \to T(X \otimes Y).\] We can strengthen this map in the second argument to obtain
\[i^t : X,TY \to T(X \otimes Y).\] Now passing back to the original monoidal category we have found a strength map $X \otimes TY \to T(X \otimes Y)$, and one can check that this satisfies the strength axioms. This derivation justifies the use of the terminology `strength' to refer to the functors 
\[\C(B_1,...,JX,...,B_n;TY) \xrightarrow{(-)^{t_i}} \C(B_1,...,TX,...,B_n;TY).\]

\begin{definition}\label{paramrelpsmonad}
(Strong relative pseudomonad) Let $\C$ and $\D$ be 2-multicategories and let $J : \D \to \C$ be a (unary) 2-functor between them. A \emph{strong relative pseudomonad} $(T,i,{}^t;\tilde t, \hat t, \theta)$ along $J$ comprises:
\begin{itemize}
    \item for every object $X$ in $\D$ an object $TX$ in $\C$ and unit map $i_X : JX \to TX$,
    \item for every $n$, index $1\leq i \leq n$, objects $B_1,...,B_{i-1},B_{i+1},...,B_n$ in $\C$ and objects $X,Y$ in $\D$ a functor \[\C(B_1,...,B_{i-1},JX,B_{i+1},...,B_n;TY) \xrightarrow{(-)^{t_i}} \C(B_1,...,B_{i-1},TX,B_{i+1},...,B_n;TY)\] called the \emph{strength} (in the $i$th argument) and which is pseudonatural in all arguments, along with three natural families of invertible 2-cells:
    \item $\tilde t_f : f \to f^{t_j}\circ_j i$,
    \item $\hat t_{f,g} : (f^{t_j} \circ_j g)^{t_{j+k-1}} \to f^{t_j} \circ_j g^{t_k}$, and
    \item $\theta_X : (i_X)^{t_1} \to 1_{TX}$
    
    for $f : B_1,...,JX,...,B_n \to TY$ and $g : C_1,...,JW,...,C_m \to TX$, satisfying the coherence conditions (1) and (2) shown below.
\end{itemize}
As a notational shorthand, when a map $f : B_1,...,JX,...,B_n \to TY$ has only one argument in the domain of the form $JX$ for some $X \in \ob\D$, we will denote its strengthening simply as $f^t$, rather than $f^{t_i}$. We will furthermore write $f^t \circ_t g$ to denote the composition of $f^t$ with $g$ in this strengthened argument. In this notation the families of invertible 2-cells above are:
\begin{align*}
    \tilde t_f &: f \to f^t \circ_t i\\
    \hat t_{f,g} &: (f^t \circ_t g)^t \to f^t \circ_t g^t\\
    \theta &: i^t \to 1
\end{align*}
(We also omit subscripts from unit maps and from $\theta$ when unambiguous.) With this notation in hand, the two coherence conditions for these 2-cells are:

\begin{enumerate}
    \item[(1)] for every $f : B_1,...JX...B_n \to TY$, $g : C_1,...,JW,...C_m \to TX$ and\\ $h : D_1,...,JV,...,D_l \to TW$ the diagram
    % https://q.uiver.app/?q=WzAsNSxbMCwwLCIoKGZedCBcXGNpcmNfdCBnKV50IFxcY2lyY190IGgpXnQiXSxbMCwxLCIoZl50IFxcY2lyY190IGdedCBcXGNpcmNfdCBoKV50Il0sWzEsMSwiZl50IFxcY2lyY190IChnXnQgXFxjaXJjX3QgaCledCJdLFsyLDAsIihmXnQgXFxjaXJjX3QgZyledCBcXGNpcmNfdCBoXnQiXSxbMiwxLCIoZl50IFxcY2lyY190IGdedCkgXFxjaXJjX3QgaF50Il0sWzAsMSwiKFxcaGF0IHRfe2YsZ30gXFxjaXJjX3QgaCledCIsMl0sWzAsMywiXFxoYXQgdF97Zl50IFxcY2lyY190IGcsaH0iXSxbMyw0LCJcXGhhdCB0X3tmLGd9IFxcY2lyY190IGhedCJdLFsxLDIsIlxcaGF0IHRfe2YsZ150IFxcY2lyY190IGh9IiwyXSxbMiw0LCJmXnQgXFxjaXJjX3QgXFxoYXQgdF97ZyxofSIsMl1d
\[\begin{tikzcd}[ampersand replacement=\&]
	{((f^t \circ_t g)^t \circ_t h)^t} \&\& {(f^t \circ_t g)^t \circ_t h^t} \\
	{(f^t \circ_t g^t \circ_t h)^t} \& {f^t \circ_t (g^t \circ_t h)^t} \& {(f^t \circ_t g^t) \circ_t h^t}
	\arrow["{(\hat t_{f,g} \circ_t h)^t}"', from=1-1, to=2-1]
	\arrow["{\hat t_{f^t \circ_t g,h}}", from=1-1, to=1-3]
	\arrow["{\hat t_{f,g} \circ_t h^t}", from=1-3, to=2-3]
	\arrow["{\hat t_{f,g^t \circ_t h}}"', from=2-1, to=2-2]
	\arrow["{f^t \circ_t \hat t_{g,h}}"', from=2-2, to=2-3]
\end{tikzcd}\] commutes, and
    \item[(2)] for every $f : B_1,...,JX,...,B_n \to TY$ the diagram
    % https://q.uiver.app/?q=WzAsNCxbMCwwLCJmXnQiXSxbMSwwLCIoZl50IFxcY2lyY190IGkpXnQiXSxbMiwwLCJmXnQgXFxjaXJjX3QgaV50Il0sWzIsMSwiZl50Il0sWzAsMSwiKFxcdGlsZGUgdF9mKV50Il0sWzEsMiwiXFxoYXQgdF9mLGkiXSxbMiwzLCJmXnQgXFxjaXJjX3QgXFx0aGV0YSJdLFswLDMsIiIsMix7ImxldmVsIjoyLCJzdHlsZSI6eyJoZWFkIjp7Im5hbWUiOiJub25lIn19fV1d
\[\begin{tikzcd}[ampersand replacement=\&]
	{f^t} \& {(f^t \circ_t i)^t} \& {f^t \circ_t i^t} \\
	\&\& {f^t}
	\arrow["{(\tilde t_f)^t}", from=1-1, to=1-2]
	\arrow["{\hat t_{f,i}}", from=1-2, to=1-3]
	\arrow["{f^t \circ_t \theta}", from=1-3, to=2-3]
	\arrow[Rightarrow, no head, from=1-1, to=2-3]
\end{tikzcd}\] commutes.
\end{enumerate}
\end{definition}

\begin{remark}
    The stipulation that the maps \[\C(B_1,...,JX,...,B_n;TY) \xrightarrow{(-)^{t_j}} \C(B_1,...,TX,...,B_n;TY)\] be pseudonatural in all arguments asks in particular for invertible 2-cells of the form
\begin{itemize}
    \item $(f \circ_k g)^t \cong f^t \circ_k g$ for $g : C_1,...,C_m \to B_k$ (where $k \neq j$).
\end{itemize} Wherever such pseudonaturality isomorphisms arise in diagrams we will leave them anonymous, as they can be inferred from the source and target.
\end{remark}

\begin{remark}
    The data for a strong relative pseudomonad resembles that for a (unary) relative pseudomonad very closely. Indeed, restricting $\C$ and $\D$ to their 2-categories of unary maps, $(T,i,{}^t)$ is exactly a (unary) relative pseudomonad, with
\begin{align*}
    (-)^* &:= (-)^t,\\
    \eta &:= \tilde t,\\
    \mu &:= \hat t,\\
    \theta &:= \theta.
\end{align*}
\end{remark}

As with relative pseudomonads, we can derive more equalities of 2-cells for a strong relative pseudomonads. The proof of the following Lemma~\ref{param3moreeqs} is formally identical to the proof of Lemma~\ref{3moreeqs}.

\begin{lemma}\label{param3moreeqs}
    Let $T$ be a strong relative pseudomonad along $J : \D \to \C$. Then the following three diagrams commute:
    \begin{enumerate}
        \item[(1)] for every $f : B_1,...,JX,...,B_n \to TY$ and $g : C_1,...,JW,...,C_m \to TX$, the diagram
        % https://q.uiver.app/?q=WzAsMyxbMCwwLCJmXnQgXFxjaXJjX3QgZyJdLFsxLDAsIihmXnQgXFxjaXJjX3QgZyledCBcXGNpcmNfdCBpIl0sWzEsMSwiZl50IFxcY2lyY190IGdedCBcXGNpcmNfdCBpIl0sWzAsMSwiXFx0aWxkZSB0X3tmXnQgXFxjaXJjX3QgZ30iXSxbMSwyLCJcXGhhdCB0X3tmLGd9IFxcY2lyY190IGkiXSxbMCwyLCJmXnQgXFxjaXJjX3QgXFx0aWxkZSB0X2ciLDJdXQ==
\[\begin{tikzcd}[ampersand replacement=\&]
	{f^t \circ_t g} \& {(f^t \circ_t g)^t \circ_t i} \\
	\& {f^t \circ_t g^t \circ_t i}
	\arrow["{\tilde t_{f^t \circ_t g}}", from=1-1, to=1-2]
	\arrow["{\hat t_{f,g} \circ_t i}", from=1-2, to=2-2]
	\arrow["{f^t \circ_t \tilde t_g}"', from=1-1, to=2-2]
\end{tikzcd}\] commutes.
        \item[(2)] for every $f : B_1,...,JX,...,B_n \to TY$, the diagram
        %% https://q.uiver.app/?q=WzAsMyxbMCwwLCIoaV50IFxcY2lyYyBmKV50Il0sWzEsMSwiZl50Il0sWzEsMCwiaV50IFxcY2lyYyBmXnQiXSxbMiwxLCJcXHRoZXRhIFxcY2lyYyBmXnQiXSxbMCwxLCIoXFx0aGV0YSBcXGNpcmMgZiledCIsMl0sWzAsMiwiXFxoYXQgdF97aSxmfSJdXQ==
\[\begin{tikzcd}[ampersand replacement=\&]
	{(i^t \circ f)^t} \& {i^t \circ f^t} \\
	\& {f^t}
	\arrow["{\theta \circ f^t}", from=1-2, to=2-2]
	\arrow["{(\theta \circ f)^t}"', from=1-1, to=2-2]
	\arrow["{\hat t_{i,f}}", from=1-1, to=1-2]
\end{tikzcd}\] commutes, and
        \item[(3)] for every object $X \in \ob\D$, the diagram
        % https://q.uiver.app/?q=WzAsMyxbMCwwLCJpIl0sWzEsMSwiaSJdLFsxLDAsImledCBcXGNpcmMgaSJdLFswLDIsIlxcdGlsZGUgdF9pIl0sWzIsMSwiXFx0aGV0YSBcXGNpcmMgaSJdLFswLDEsIiIsMix7ImxldmVsIjoyLCJzdHlsZSI6eyJoZWFkIjp7Im5hbWUiOiJub25lIn19fV1d
\[\begin{tikzcd}[ampersand replacement=\&]
	i \& {i^t \circ i} \\
	\& i
	\arrow["{\tilde t_i}", from=1-1, to=1-2]
	\arrow["{\theta \circ i}", from=1-2, to=2-2]
	\arrow[Rightarrow, no head, from=1-1, to=2-2]
\end{tikzcd}\] commutes.
    \end{enumerate}
\end{lemma}

\begin{example}
    The presheaf relative pseudomonad from Example~\ref{psh ex} can be given the structure of a strong relative pseudomonad. Given a multimorphism \[f : B_1,...,X,...,B_n \to \Psh Y\] with $X,Y \in \Cat$ and $B_k \in \CAT$, its strengthening $f^t$ is defined to be the left Kan extension
% https://q.uiver.app/?q=WzAsMyxbMCwwLCJCXzEsLi4uLFgsLi4uLEJfbiJdLFsyLDAsIkJfMSwuLi4sXFxQc2ggWCwuLi4sQl9uIl0sWzIsMiwiXFxQc2ggWSJdLFswLDIsImYiLDJdLFswLDEsIjEsLi4uLHlfWCwuLi4sMSJdLFsxLDIsImZedCA6PSBcXExhbl97MSwuLi4seSwuLi4sMX0gZiJdLFszLDUsIlxcdGlsZGUgdF9mIiwwLHsic2hvcnRlbiI6eyJzb3VyY2UiOjIwLCJ0YXJnZXQiOjIwfX1dXQ==
\[\begin{tikzcd}[ampersand replacement=\&]
	{B_1,...,X,...,B_n} \&\& {B_1,...,\Psh X,...,B_n} \\
	\\
	\&\& {\Psh Y}
	\arrow[""{name=0, anchor=center, inner sep=0}, "f"', from=1-1, to=3-3]
	\arrow["{1,...,y_X,...,1}", from=1-1, to=1-3]
	\arrow[""{name=1, anchor=center, inner sep=0}, "{f^t := \Lan_{1,...,y,...,1} f}", from=1-3, to=3-3]
	\arrow["{\tilde t_f}", shorten <=11pt, shorten >=11pt, Rightarrow, from=0, to=1]
\end{tikzcd}\] which also defines the 2-cells $\tilde t_f : f \to f^t \circ_t y$. As when giving $\Psh$ a relative pseudomonad structure, the 2-cells $\hat t_{f,g}$, $\theta$ are defined via the universal property of the left Kan extension. For details and a proof that this indeed endows $\Psh$ with a strong relative pseudomonad structure, see Proposition~\ref{pshisparam} in the final section.
\end{example}

Having generalised Kock's notion of a strong monad, we seek to prove a generalisation of his result that every strong monad is a lax monoidal functor. For this we define a notion of a pseudo-multifunctor on a 2-multicategory.

\begin{definition}
    (Pseudo-multifunctor) Given multi-2-categories $\C,\D$, a \emph{pseudo-multifunctor} $F : \D \to \C$ consists of:
    \begin{itemize}
        \item a function $\ob\D \xrightarrow{F} \ob\C : X \mapsto FX$,
        \item for each hom-category $\D(X_1,...,X_n;Y)$ in $\D$ a functor \[\D(X_1,...,X_n;Y) \to \C(FX_1,...,FX_n;FY) : f \mapsto Ff,\] along with
        \item for each $X \in \ob\D$ an invertible 2-cell \[\tilde F_X : F1_X \implies 1_{FX},\]
        \item for each $f : X_1,...,X_n \to Y$, $1 \leq i \leq n$ and $g : W_1,...,W_m \to X_i$ an invertible 2-cell \[\hat F_{f,g} : F(f \circ_i g) \implies Ff \circ_i Fg\]
    \end{itemize} satisfying the following three coherence conditions which parallel the unit and associativity diagrams for a lax monoidal functor:
    \begin{enumerate}
        \item [(1),(2)] two unit axioms: for each $f : X_1,...,X_n \to Y$ and $1 \leq i \leq n$ the diagrams
        % https://q.uiver.app/?q=WzAsOCxbMCwxLCJGZiJdLFswLDAsIkYoMV9ZXFxjaXJjIGYpIl0sWzEsMSwiMV97Rll9IFxcY2lyYyBGZiJdLFsxLDAsIkYxX1lcXGNpcmMgRmYiXSxbMiwxLCJGZiJdLFsyLDAsIkYoZiBcXGNpcmNfaSAxX3tYX2l9KSJdLFszLDEsIkZmIFxcY2lyY19pIDFfe0ZYX2l9Il0sWzMsMCwiRmYgXFxjaXJjX2kgRjFfe1hfaX0iXSxbMCwxLCIiLDAseyJsZXZlbCI6Miwic3R5bGUiOnsiaGVhZCI6eyJuYW1lIjoibm9uZSJ9fX1dLFswLDIsIiIsMix7ImxldmVsIjoyLCJzdHlsZSI6eyJoZWFkIjp7Im5hbWUiOiJub25lIn19fV0sWzEsMywiXFxoYXQgRl97MSxmfSJdLFszLDIsIlxcdGlsZGUgRl9ZIFxcY2lyYyBGZiJdLFs1LDcsIlxcaGF0IEZfe2YsMX0iXSxbNyw2LCJGZiBcXGNpcmNfaSBcXHRpbGRlIEZfe1hfaX0iXSxbNSw0LCIiLDIseyJsZXZlbCI6Miwic3R5bGUiOnsiaGVhZCI6eyJuYW1lIjoibm9uZSJ9fX1dLFs0LDYsIiIsMix7ImxldmVsIjoyLCJzdHlsZSI6eyJoZWFkIjp7Im5hbWUiOiJub25lIn19fV1d
\[\begin{tikzcd}[ampersand replacement=\&]
	{F(1_Y\circ f)} \& {F1_Y\circ Ff} \& {F(f \circ_i 1_{X_i})} \& {Ff \circ_i F1_{X_i}} \\
	Ff \& {1_{FY} \circ Ff} \& Ff \& {Ff \circ_i 1_{FX_i}}
	\arrow[Rightarrow, no head, from=2-1, to=1-1]
	\arrow[Rightarrow, no head, from=2-1, to=2-2]
	\arrow["{\hat F_{1,f}}", from=1-1, to=1-2]
	\arrow["{\tilde F_Y \circ Ff}", from=1-2, to=2-2]
	\arrow["{\hat F_{f,1}}", from=1-3, to=1-4]
	\arrow["{Ff \circ_i \tilde F_{X_i}}", from=1-4, to=2-4]
	\arrow[Rightarrow, no head, from=1-3, to=2-3]
	\arrow[Rightarrow, no head, from=2-3, to=2-4]
\end{tikzcd}\] commute, and
    \item [(3)] one associativity axiom: for each  $f : X_1,...,X_n \to Y$, $1 \leq i \leq n$, $g : W_1,...,W_m \to X_i$, $1 \leq j \leq m$ and $h : V_1,...,V_l \to W_j$ the diagram
    \begin{small}
        % https://q.uiver.app/?q=WzAsNixbMCwwLCJGKGYgXFxjaXJjX2kgKGcgXFxjaXJjX2ogaCkpIl0sWzAsMSwiRigoZiBcXGNpcmNfaSBnKSBcXGNpcmNfe2krai0xfSBoKSJdLFsyLDAsIkZmIFxcY2lyY19pIEYoZyBcXGNpcmNfaiBoKSJdLFsyLDEsIkYoZiBcXGNpcmNfaSBnKSBcXGNpcmNfe2krai0xfSBGaCJdLFs0LDAsIkZmIFxcY2lyY19pIChGZyBcXGNpcmNfaiBGaCkiXSxbNCwxLCIoRmYgXFxjaXJjX2kgRmcpIFxcY2lyY197aStqLTF9IEZoIl0sWzAsMSwiIiwyLHsibGV2ZWwiOjIsInN0eWxlIjp7ImhlYWQiOnsibmFtZSI6Im5vbmUifX19XSxbMCwyLCJcXGhhdCBGX3tmLGcgXFxjaXJjX2ogaH0iXSxbMSwzLCJcXGhhdCBGX3tmIFxcY2lyY19pIGcsaH0iLDJdLFsyLDQsIkZmIFxcY2lyYyBfaSBcXGhhdCBGX3tnLGh9Il0sWzMsNSwiXFxoYXQgRl97ZixnfSBcXGNpcmNfe2krai0xfSBGaCIsMl0sWzQsNSwiIiwwLHsibGV2ZWwiOjIsInN0eWxlIjp7ImhlYWQiOnsibmFtZSI6Im5vbmUifX19XV0=
\[\begin{tikzcd}[ampersand replacement=\&]
	{F(f \circ_i (g \circ_j h))} \&\& {Ff \circ_i F(g \circ_j h)} \&\& {Ff \circ_i (Fg \circ_j Fh)} \\
	{F((f \circ_i g) \circ_{i+j-1} h)} \&\& {F(f \circ_i g) \circ_{i+j-1} Fh} \&\& {(Ff \circ_i Fg) \circ_{i+j-1} Fh}
	\arrow[Rightarrow, no head, from=1-1, to=2-1]
	\arrow["{\hat F_{f,g \circ_j h}}", from=1-1, to=1-3]
	\arrow["{\hat F_{f \circ_i g,h}}"', from=2-1, to=2-3]
	\arrow["{Ff \circ _i \hat F_{g,h}}", from=1-3, to=1-5]
	\arrow["{\hat F_{f,g} \circ_{i+j-1} Fh}"', from=2-3, to=2-5]
	\arrow[Rightarrow, no head, from=1-5, to=2-5]
\end{tikzcd}\]
    \end{small}
     commutes.
    \end{enumerate}
If the 2-cells $\tilde F$, $\hat F$ are all identities we call $F$ a \emph{(strict) multicategorical 2-functor}.
\end{definition}

Just as the underlying functor of every strong monad is a lax monoidal functor, the underlying pseudofunctor of every strong relative pseudomonad is a pseudo-multifunctor.

\begin{proposition}\label{param=>psfunctor}
    Let $T$ be a strong relative pseudomonad along multicategorical 2-functor $J : \D \to \C$. Then $T$ is a pseudo-multifunctor $T : \D \to \C$.
\end{proposition}
\begin{proof}
Suppose $T$ is a strong relative pseudomonad. As a point of notation, given a map $f : X_1,...,X_n \to Y$ let us define
\[\bar f := i_Y \circ JF : JX_1,...,JX_n \to TY.\] Now to show the $T$ is a pseudo-multifunctor, we begin by defining the action of $T$ on 1-cells by the functors
\[\D(X_1,...,X_n;Y) \xrightarrow{(i_Y\circ J-)^{t_1t_2...t_n}} \C(TX_1,...,TX_n;TY),\] so that for $f : X_1,...,X_n \to Y$ we have \[Tf := (i_Y\circ Jf)^{t_1t_2...t_n} = \bar f^{t_1,...,t_n} : TX_1,...,TX_n \to TY.\]
 
We need to construct 2-cells $\tilde T_X : T1_X \implies 1_{TX}$ and $\hat T_{f,g} : T(f \circ_i g) \implies Tf \circ_i Tg$. For the former, we can use the map
% https://q.uiver.app/?q=WzAsMixbMCwwLCJUMV9YID0gKGlfWCBcXGNpcmMgSjFfWCledCA9IChpX1gpXnQiXSxbMSwwLCIxX3tUWH0iXSxbMCwxLCJcXHRoZXRhX1giXV0=
\[\begin{tikzcd}[ampersand replacement=\&]
	{T1_X = (i_X \circ J1_X)^t = (i_X)^t} \& {1_{TX}}
	\arrow["{\theta_X}", from=1-1, to=1-2]
\end{tikzcd}\] and for the latter, we employ the composite
\begin{align*}
    T(f \circ_i g) &= (i \circ (Jf \circ_i Jg))^{t_1...t_{n+m-1}} = (\bar f \circ_i Jg)^{t_1...t_{n+m-1}}\\
    &\xrightarrow{\sim} (\bar f^{t_1...t_{i-1}} \circ_i Jg)^{t_i...t_{n+m-1}}\\
    &\xrightarrow{\tilde t} (\bar f^{t_1...t_i} \circ_i \bar g)^{t_i...t_{n+m-1}}\\
    &\xrightarrow{\hat t...\hat t} (\bar f^{t_1...t_i} \circ_i \bar g^{t_1...t_m})^{t_{i+m}...t_{n+m-1}}\\
    &\xrightarrow{\sim} \bar f^{t_1...t_n} \circ_i \bar g^{t_1...t_m} = Tf \circ_i Tg.\\
\end{align*}
It remains to show that the three coherence conditions hold. For the first
% https://q.uiver.app/?q=WzAsNCxbMCwxLCJUZiJdLFswLDAsIlQoMV9ZXFxjaXJjIGYpIl0sWzEsMSwiMV97VFl9IFxcY2lyYyBUZiJdLFsxLDAsIlQxX1lcXGNpcmMgVGYiXSxbMCwxLCIiLDAseyJsZXZlbCI6Miwic3R5bGUiOnsiaGVhZCI6eyJuYW1lIjoibm9uZSJ9fX1dLFswLDIsIiIsMix7ImxldmVsIjoyLCJzdHlsZSI6eyJoZWFkIjp7Im5hbWUiOiJub25lIn19fV0sWzEsMywiXFxoYXQgVF97MSxmfSJdLFszLDIsIlxcdGlsZGUgVF9ZIFxcY2lyYyBUZiJdXQ==
\[\begin{tikzcd}[ampersand replacement=\&]
	{T(1_Y\circ f)} \& {T1_Y\circ Tf} \\
	Tf \& {1_{TY} \circ Tf}
	\arrow[Rightarrow, no head, from=2-1, to=1-1]
	\arrow[Rightarrow, no head, from=2-1, to=2-2]
	\arrow["{\hat T_{1,f}}", from=1-1, to=1-2]
	\arrow["{\tilde T_Y \circ Tf}", from=1-2, to=2-2]
\end{tikzcd}\] we rewrite everything in terms of parameterisation and obtain the diagram
% https://q.uiver.app/?q=WzAsNixbMCwyLCIoaSBcXGNpcmMgSmYpXnt0XzF0XzIuLi50X259Il0sWzAsMCwiKFxcYmFyIDEgXFxjaXJjIEpmKV57dF8xdF8yLi4udF9ufSJdLFsyLDIsIlxcYmFyIGZee3RfMXRfMi4uLnRfbn0iXSxbMiwwLCJcXGJhciAxXnQgXFxjaXJjIFxcYmFyIGZee3RfMS4uLnRfbn0iXSxbMiwxLCJpXnQgXFxjaXJjIFxcYmFyIGZee3RfMS4uLnRfbn0iXSxbMSwwLCIoXFxiYXIgMV50IFxcY2lyYyBcXGJhciBmKV57dF8xdF8yLi4udF9ufSJdLFswLDEsIiIsMCx7ImxldmVsIjoyLCJzdHlsZSI6eyJoZWFkIjp7Im5hbWUiOiJub25lIn19fV0sWzAsMiwiIiwyLHsibGV2ZWwiOjIsInN0eWxlIjp7ImhlYWQiOnsibmFtZSI6Im5vbmUifX19XSxbMyw0LCIiLDAseyJsZXZlbCI6Miwic3R5bGUiOnsiaGVhZCI6eyJuYW1lIjoibm9uZSJ9fX1dLFs0LDIsIlxcdGhldGEiXSxbMSw1LCJcXHRpbGRlIHQiXSxbNSwzLCJcXGhhdCB0Li4uXFxoYXQgdCJdXQ==
\[\begin{tikzcd}[ampersand replacement=\&]
	{(\bar 1 \circ Jf)^{t_1t_2...t_n}} \& {(\bar 1^t \circ \bar f)^{t_1t_2...t_n}} \& {\bar 1^t \circ \bar f^{t_1...t_n}} \\
	\&\& {i^t \circ \bar f^{t_1...t_n}} \\
	{(i \circ Jf)^{t_1t_2...t_n}} \&\& {\bar f^{t_1t_2...t_n}}
	\arrow[Rightarrow, no head, from=3-1, to=1-1]
	\arrow[Rightarrow, no head, from=3-1, to=3-3]
	\arrow[Rightarrow, no head, from=1-3, to=2-3]
	\arrow["\theta", from=2-3, to=3-3]
	\arrow["{\tilde t}", from=1-1, to=1-2]
	\arrow["{\hat t...\hat t}", from=1-2, to=1-3]
\end{tikzcd}\] To show that this commutes, we fill it in
% https://q.uiver.app/?q=WzAsNyxbMCwxLCIoaSBcXGNpcmMgSmYpXnt0XzF0XzIuLi50X259Il0sWzAsMCwiKFxcYmFyIDEgXFxjaXJjIEpmKV57dF8xdF8yLi4udF9ufSJdLFsxLDIsIlxcYmFyIGZee3RfMXRfMi4uLnRfbn0iXSxbMiwwLCJcXGJhciAxXnQgXFxjaXJjIFxcYmFyIGZee3RfMS4uLnRfbn0iXSxbMiwxLCJpXnQgXFxjaXJjIFxcYmFyIGZee3RfMS4uLnRfbn0iXSxbMSwwLCIoXFxiYXIgMV50IFxcY2lyYyBcXGJhciBmKV57dF8xdF8yLi4udF9ufSJdLFsxLDEsIihpXnQgXFxjaXJjIFxcYmFyIGYpXnt0XzF0XzIuLi50X259Il0sWzAsMSwiIiwwLHsibGV2ZWwiOjIsInN0eWxlIjp7ImhlYWQiOnsibmFtZSI6Im5vbmUifX19XSxbMCwyLCIiLDIseyJsZXZlbCI6Miwic3R5bGUiOnsiaGVhZCI6eyJuYW1lIjoibm9uZSJ9fX1dLFszLDQsIiIsMCx7ImxldmVsIjoyLCJzdHlsZSI6eyJoZWFkIjp7Im5hbWUiOiJub25lIn19fV0sWzQsMiwiXFx0aGV0YSJdLFsxLDUsIlxcdGlsZGUgdCJdLFs1LDMsIlxcaGF0IHQuLi5cXGhhdCB0Il0sWzUsNiwiIiwwLHsibGV2ZWwiOjIsInN0eWxlIjp7ImhlYWQiOnsibmFtZSI6Im5vbmUifX19XSxbMCw2LCJcXHRpbGRlIHQiXSxbNiw0LCJcXGhhdCB0Li4uXFxoYXQgdCJdLFs2LDIsIlxcdGhldGEiXV0=
\[\begin{tikzcd}[ampersand replacement=\&]
	{(\bar 1 \circ Jf)^{t_1t_2...t_n}} \& {(\bar 1^t \circ \bar f)^{t_1t_2...t_n}} \& {\bar 1^t \circ \bar f^{t_1...t_n}} \\
	{(i \circ Jf)^{t_1t_2...t_n}} \& {(i^t \circ \bar f)^{t_1t_2...t_n}} \& {i^t \circ \bar f^{t_1...t_n}} \\
	\& {\bar f^{t_1t_2...t_n}}
	\arrow[Rightarrow, no head, from=2-1, to=1-1]
	\arrow[Rightarrow, no head, from=2-1, to=3-2]
	\arrow[Rightarrow, no head, from=1-3, to=2-3]
	\arrow["\theta", from=2-3, to=3-2]
	\arrow["{\tilde t}", from=1-1, to=1-2]
	\arrow["{\hat t...\hat t}", from=1-2, to=1-3]
	\arrow[Rightarrow, no head, from=1-2, to=2-2]
	\arrow["{\tilde t}", from=2-1, to=2-2]
	\arrow["{\hat t...\hat t}", from=2-2, to=2-3]
	\arrow["\theta", from=2-2, to=3-2]
\end{tikzcd}\] with two naturality squares and equalities of 2-cells (3) and (2) from Lemma~\ref{param3moreeqs}.

For the second
% https://q.uiver.app/?q=WzAsNCxbMCwxLCJUZiJdLFswLDAsIlQoZiBcXGNpcmNfaSAxX3tYX2l9KSJdLFsxLDEsIlRmIFxcY2lyY19pIDFfe1RYX2l9Il0sWzEsMCwiVGYgXFxjaXJjX2kgVDFfe1hfaX0iXSxbMSwzLCJcXGhhdCBUX3tmLDF9Il0sWzMsMiwiVGYgXFxjaXJjX2kgXFx0aWxkZSBUX3tYX2l9Il0sWzEsMCwiIiwyLHsibGV2ZWwiOjIsInN0eWxlIjp7ImhlYWQiOnsibmFtZSI6Im5vbmUifX19XSxbMCwyLCIiLDIseyJsZXZlbCI6Miwic3R5bGUiOnsiaGVhZCI6eyJuYW1lIjoibm9uZSJ9fX1dXQ==
\[\begin{tikzcd}[ampersand replacement=\&]
	{T(f \circ_i 1_{X_i})} \& {Tf \circ_i T1_{X_i}} \\
	Tf \& {Tf \circ_i 1_{TX_i}}
	\arrow["{\hat T_{f,1}}", from=1-1, to=1-2]
	\arrow["{Tf \circ_i \tilde T_{X_i}}", from=1-2, to=2-2]
	\arrow[Rightarrow, no head, from=1-1, to=2-1]
	\arrow[Rightarrow, no head, from=2-1, to=2-2]
\end{tikzcd}\] we rewrite everything in terms of parameterisation and obtain the diagram
% https://q.uiver.app/?q=WzAsOCxbMCw0LCJcXGJhciBmXnt0XzEuLi50X259Il0sWzAsMCwiKFxcYmFyIGYgXFxjaXJjX2kgSjEpXnt0XzEuLi50X259Il0sWzIsNCwiXFxiYXIgZl57dF8xLi4udF9ufSBcXGNpcmNfaSAxIl0sWzIsMiwiXFxiYXIgZl57dF8xLi4udF9ufSBcXGNpcmNfaSBcXGJhciAxXnQiXSxbMiwzLCJcXGJhciBmXnt0XzEuLi50X259IFxcY2lyY19pIGledCJdLFsxLDAsIihcXGJhciBmXnt0XzEuLi50X3tpLTF9fSBcXGNpcmNfaSBKMSlee3RfaS4uLnRfbn0iXSxbMiwwLCIoXFxiYXIgZl57dF8xLi4udF9pfSBcXGNpcmNfaSBcXGJhciAxKV57dF9pLi4udF9ufSJdLFsyLDEsIihcXGJhciBmXnt0XzEuLi50X2l9IFxcY2lyY19pIFxcYmFyIDFedClee3Rfe2krMX0uLi50X259Il0sWzEsMCwiIiwyLHsibGV2ZWwiOjIsInN0eWxlIjp7ImhlYWQiOnsibmFtZSI6Im5vbmUifX19XSxbMCwyLCIiLDIseyJsZXZlbCI6Miwic3R5bGUiOnsiaGVhZCI6eyJuYW1lIjoibm9uZSJ9fX1dLFszLDQsIiIsMCx7ImxldmVsIjoyLCJzdHlsZSI6eyJoZWFkIjp7Im5hbWUiOiJub25lIn19fV0sWzQsMiwiXFx0aGV0YSJdLFsxLDUsIlxcc2ltIl0sWzUsNiwiXFx0aWxkZSB0Il0sWzYsNywiXFxoYXQgdCJdLFs3LDMsIlxcc2ltIl1d
\[\begin{tikzcd}[ampersand replacement=\&]
	{(\bar f \circ_i J1)^{t_1...t_n}} \& {(\bar f^{t_1...t_{i-1}} \circ_i J1)^{t_i...t_n}} \& {(\bar f^{t_1...t_i} \circ_i \bar 1)^{t_i...t_n}} \\
	\&\& {(\bar f^{t_1...t_i} \circ_i \bar 1^t)^{t_{i+1}...t_n}} \\
	\&\& {\bar f^{t_1...t_n} \circ_i \bar 1^t} \\
	\&\& {\bar f^{t_1...t_n} \circ_i i^t} \\
	{\bar f^{t_1...t_n}} \&\& {\bar f^{t_1...t_n} \circ_i 1}
	\arrow[Rightarrow, no head, from=1-1, to=5-1]
	\arrow[Rightarrow, no head, from=5-1, to=5-3]
	\arrow[Rightarrow, no head, from=3-3, to=4-3]
	\arrow["\theta", from=4-3, to=5-3]
	\arrow["\sim", from=1-1, to=1-2]
	\arrow["{\tilde t}", from=1-2, to=1-3]
	\arrow["{\hat t}", from=1-3, to=2-3]
	\arrow["\sim", from=2-3, to=3-3]
\end{tikzcd}\] To show that this commutes, we fill it in
% https://q.uiver.app/?q=WzAsMTEsWzEsMCwiKFxcYmFyIGYgXFxjaXJjX2kgSjEpXnt0XzEuLi50X259Il0sWzAsNCwiXFxiYXIgZl57dF8xLi4udF9ufSBcXGNpcmNfaSAxIl0sWzIsNCwiXFxiYXIgZl57dF8xLi4udF9ufSBcXGNpcmNfaSBcXGJhciAxXnQiXSxbMSw0LCJcXGJhciBmXnt0XzEuLi50X259IFxcY2lyY19pIGledCJdLFsyLDEsIihcXGJhciBmXnt0XzEuLi50X3tpLTF9fSBcXGNpcmNfaSBKMSlee3RfaS4uLnRfbn0iXSxbMiwyLCIoXFxiYXIgZl57dF8xLi4udF9pfSBcXGNpcmNfaSBcXGJhciAxKV57dF9pLi4udF9ufSJdLFsyLDMsIihcXGJhciBmXnt0XzEuLi50X2l9IFxcY2lyY19pIFxcYmFyIDFedClee3Rfe2krMX0uLi50X259Il0sWzEsMywiKFxcYmFyIGZee3RfMS4uLnRfaX0gXFxjaXJjX2kgaV50KV57dF97aSsxfS4uLnRfbn0iXSxbMCwzLCIoXFxiYXIgZl57dF8xLi4udF9pfSBcXGNpcmNfaSAxKV57dF97aSsxfS4uLnRfbn0iXSxbMSwyLCIoXFxiYXIgZl57dF8xLi4udF9pfSBcXGNpcmNfaSBpKV57dF9pLi4udF9ufSJdLFsxLDEsIlxcYmFyIGZee3RfMS4uLnRfbn0iXSxbMiwzLCIiLDAseyJsZXZlbCI6Miwic3R5bGUiOnsiaGVhZCI6eyJuYW1lIjoibm9uZSJ9fX1dLFszLDEsIlxcdGhldGEiXSxbMCw0LCJcXHNpbSJdLFs0LDUsIlxcdGlsZGUgdCJdLFs1LDYsIlxcaGF0IHQiXSxbNiwyLCJcXHNpbSJdLFs2LDcsIiIsMCx7ImxldmVsIjoyLCJzdHlsZSI6eyJoZWFkIjp7Im5hbWUiOiJub25lIn19fV0sWzcsMywiXFxzaW0iXSxbNyw4LCJcXHRoZXRhIl0sWzgsMSwiXFxzaW0iXSxbNSw5LCIiLDAseyJsZXZlbCI6Miwic3R5bGUiOnsiaGVhZCI6eyJuYW1lIjoibm9uZSJ9fX1dLFs5LDcsIlxcaGF0IHQiXSxbNCwxMCwiIiwwLHsibGV2ZWwiOjIsInN0eWxlIjp7ImhlYWQiOnsibmFtZSI6Im5vbmUifX19XSxbMCwxMCwiIiwyLHsibGV2ZWwiOjIsInN0eWxlIjp7ImhlYWQiOnsibmFtZSI6Im5vbmUifX19XSxbMTAsOSwiXFx0aWxkZSB0Il0sWzEwLDgsIiIsMCx7ImxldmVsIjoyLCJzdHlsZSI6eyJoZWFkIjp7Im5hbWUiOiJub25lIn19fV1d
\[\begin{tikzcd}[ampersand replacement=\&]
	\& {(\bar f \circ_i J1)^{t_1...t_n}} \\
	\& {\bar f^{t_1...t_n}} \& {(\bar f^{t_1...t_{i-1}} \circ_i J1)^{t_i...t_n}} \\
	\& {(\bar f^{t_1...t_i} \circ_i i)^{t_i...t_n}} \& {(\bar f^{t_1...t_i} \circ_i \bar 1)^{t_i...t_n}} \\
	{(\bar f^{t_1...t_i} \circ_i 1)^{t_{i+1}...t_n}} \& {(\bar f^{t_1...t_i} \circ_i i^t)^{t_{i+1}...t_n}} \& {(\bar f^{t_1...t_i} \circ_i \bar 1^t)^{t_{i+1}...t_n}} \\
	{\bar f^{t_1...t_n} \circ_i 1} \& {\bar f^{t_1...t_n} \circ_i i^t} \& {\bar f^{t_1...t_n} \circ_i \bar 1^t}
	\arrow[Rightarrow, no head, from=5-3, to=5-2]
	\arrow["\theta", from=5-2, to=5-1]
	\arrow["\sim", from=1-2, to=2-3]
	\arrow["{\tilde t}", from=2-3, to=3-3]
	\arrow["{\hat t}", from=3-3, to=4-3]
	\arrow["\sim", from=4-3, to=5-3]
	\arrow[Rightarrow, no head, from=4-3, to=4-2]
	\arrow["\sim", from=4-2, to=5-2]
	\arrow["\theta", from=4-2, to=4-1]
	\arrow["\sim", from=4-1, to=5-1]
	\arrow[Rightarrow, no head, from=3-3, to=3-2]
	\arrow["{\hat t}", from=3-2, to=4-2]
	\arrow[Rightarrow, no head, from=2-3, to=2-2]
	\arrow[Rightarrow, no head, from=1-2, to=2-2]
	\arrow["{\tilde t}", from=2-2, to=3-2]
	\arrow[Rightarrow, no head, from=2-2, to=4-1]
\end{tikzcd}\] with naturality squares and the equality of 2-cells (2) from Definition~\ref{paramrelpsmonad}.

Finally, for the third diagram
\begin{small}
    % https://q.uiver.app/?q=WzAsNixbMCwwLCJUKGYgXFxjaXJjX2kgKGcgXFxjaXJjX2ogaCkpIl0sWzAsMSwiVCgoZiBcXGNpcmNfaSBnKSBcXGNpcmNfe2krai0xfSBoKSJdLFsyLDAsIlRmIFxcY2lyY19pIFQoZyBcXGNpcmNfaiBoKSJdLFsyLDEsIlQoZiBcXGNpcmNfaSBnKSBcXGNpcmNfe2krai0xfSBUaCJdLFs0LDAsIlRmIFxcY2lyY19pIChUZyBcXGNpcmNfaiBUaCkiXSxbNCwxLCIoVGYgXFxjaXJjX2kgVGcpIFxcY2lyY197aStqLTF9IFRoIl0sWzAsMSwiIiwyLHsibGV2ZWwiOjIsInN0eWxlIjp7ImhlYWQiOnsibmFtZSI6Im5vbmUifX19XSxbMCwyLCJcXGhhdCBUX3tmLGcgXFxjaXJjX2ogaH0iXSxbMSwzLCJcXGhhdCBUX3tmIFxcY2lyY19pIGcsaH0iLDJdLFsyLDQsIlRmIFxcY2lyYyBfaSBcXGhhdCBUX3tnLGh9Il0sWzMsNSwiXFxoYXQgVF97ZixnfSBcXGNpcmNfe2krai0xfSBUaCIsMl0sWzQsNSwiIiwwLHsibGV2ZWwiOjIsInN0eWxlIjp7ImhlYWQiOnsibmFtZSI6Im5vbmUifX19XV0=
\[\begin{tikzcd}[ampersand replacement=\&]
	{T(f \circ_i (g \circ_j h))} \&\& {Tf \circ_i T(g \circ_j h)} \&\& {Tf \circ_i (Tg \circ_j Th)} \\
	{T((f \circ_i g) \circ_{i+j-1} h)} \&\& {T(f \circ_i g) \circ_{i+j-1} Th} \&\& {(Tf \circ_i Tg) \circ_{i+j-1} Th}
	\arrow[Rightarrow, no head, from=1-1, to=2-1]
	\arrow["{\hat T_{f,g \circ_j h}}", from=1-1, to=1-3]
	\arrow["{\hat T_{f \circ_i g,h}}"', from=2-1, to=2-3]
	\arrow["{Tf \circ _i \hat T_{g,h}}", from=1-3, to=1-5]
	\arrow["{\hat T_{f,g} \circ_{i+j-1} Th}"', from=2-3, to=2-5]
	\arrow[Rightarrow, no head, from=1-5, to=2-5]
\end{tikzcd}\]
\end{small}
 in the interest of space we shall merely note that verification involves, aside from naturality squares, only the equality of 2-cells (1) from Lemma~\ref{param3moreeqs}. Thus, with these three coherence conditions, every strong relative pseudomonad is indeed a pseudo-multifunctor.
\end{proof}

\begin{example}\label{psh is multi psfun}
    Proposition~\ref{param=>psfunctor} will imply that the presheaf relative pseudomonad is a pseudo-multifunctor. Using the coend formula for the left Kan extension, we find that for example, given a functor $F : A \times B \times C \to D$ in $\Cat$, the multicategorical action of $\Psh$ on $F$ has the form
    \begin{align*}
        \Psh F : \Psh A \times \Psh B \times \Psh C &\to \Psh D\\
        (p,q,r) &\mapsto \int^{c}\int^{b}\int^{a} p(a)\times q(b) \times r(c) \times y_{F(a,b,c)}.
    \end{align*}
\end{example}

\section{Pseudocommutativity}

In the classical situation described in \cite{kock1970}, a strong monad with left-strength $s$ and right-strength $t$ can be given the structure of lax monoidal functor in two ways:
\begin{align*}
    TX \otimes TY &\xrightarrow{t} T(TX \otimes Y) \xrightarrow{Ts} TT(X \otimes Y) \xrightarrow{\mu} T(X \otimes Y)\\
    TX \otimes TY &\xrightarrow{s} T(X \otimes TY) \xrightarrow{Tt} TT(X \otimes Y) \xrightarrow{\mu} T(X \otimes Y)\\
\end{align*} It is then natural to ask about those strong monads for which these two composites are equal, which Kock called \emph{commutative monads}. Hyland and Power \cite{hp2002} extend this notion to the 2-categorical setting, defining \emph{pseudocommutativity} by asking only for an invertible 2-cell between the two composites. 

Analogously, there is some freedom in the pseudo-multifunctorial structure we place on a given strong relative pseudomonad $T$; we defined the action of $T$ on morphisms by
\[Tf := \bar f^{t_1...t_n},\] but we could equally well have chosen
\[Tf := \bar f^{t_n...t_1}\] with the strengthenings applied in the reverse order. We define pseudocommutativity in our more general setting to imply that the two choices of definition of $Tf$ are coherently isomorphic.

\begin{definition}\label{pscom}
    (Pseudocommutative monad) Let $T$ be a strong relative pseudomonad. We say that $T$ is \emph{pseudocommutative} if for every pair of indices $1\leq j<k\leq n$ and map \[f : B_1,...,B_{j-1},JX,B_{j+1}...,B_{k-1},JY,B_{k+1},...,B_n \to TZ\] we have an invertible 2-cell \[\gamma_f : f^{t_kt_j} \to f^{t_jt_k} : B_1,...,TX,...,TY,...,B_n \to TZ\] which is pseudonatural in all arguments and which satisfies five coherence conditions (two for $\tilde t$, two for $\hat t$, and a braiding condition).
    
    We will extend our notation in the following way. When a map \[f : B_1,...,JX,...,JY,...,B_n \to TZ\] has two explicitly possible strengthenings, let strengthening in the leftmost of these two arguments be denoted by $f^s$ with 2-cells $\tilde s : f \to f^s \circ_s i$ and $\hat s : (f^s \circ_s g)^s \to f^s \circ g^t$, and let strengthening in the rightmost of these two arguments be denoted by $f^t$ with 2-cells $\tilde t$, $\hat t$. When $f$ has three explicitly possible strengthenings we furthermore use $f^u$, etc. The coherence conditions $\gamma$ must satisfy are as follows:
    \begin{enumerate}
        \item [(1),\ (2)] Precomposing $\gamma_f$ in the $j$th or $k$th argument with a unit map $i$: the diagrams
        % https://q.uiver.app/?q=WzAsOCxbMiwwLCJmXnMiXSxbMywwLCIoZl50XFxjaXJjX3QgaSlecyJdLFszLDEsImZee3RzfVxcY2lyY190IGkiXSxbMywyLCJmXntzdH1cXGNpcmNfdCBpIl0sWzAsMCwiZl50Il0sWzEsMiwiKGZecyBcXGNpcmNfcyBpKV50Il0sWzEsMSwiZl57c3R9IFxcY2lyY19zIGkiXSxbMSwwLCJmXnt0c30gXFxjaXJjX3MgaSJdLFswLDEsIihcXHRpbGRlIHRfZilecyJdLFsxLDIsIlxcc2ltIl0sWzIsMywiXFxnYW1tYV9mXFxjaXJjX3QgaSJdLFswLDMsIlxcdGlsZGUgdF97Zl5zfSIsMl0sWzQsNSwiKFxcdGlsZGUgc19mKV50IiwyXSxbNCw3LCJcXHRpbGRlIHNfe2ZedH0iXSxbNiw1LCJcXHNpbSJdLFs3LDYsIlxcZ2FtbWFfZiBcXGNpcmNfcyBpIl1d
\[\begin{tikzcd}[ampersand replacement=\&]
	{f^t} \& {f^{ts} \circ_s i} \& {f^s} \& {(f^t\circ_t i)^s} \\
	\& {f^{st} \circ_s i} \&\& {f^{ts}\circ_t i} \\
	\& {(f^s \circ_s i)^t} \&\& {f^{st}\circ_t i}
	\arrow["{(\tilde t_f)^s}", from=1-3, to=1-4]
	\arrow["\sim", from=1-4, to=2-4]
	\arrow["{\gamma_f\circ_t i}", from=2-4, to=3-4]
	\arrow["{\tilde t_{f^s}}"', from=1-3, to=3-4]
	\arrow["{(\tilde s_f)^t}"', from=1-1, to=3-2]
	\arrow["{\tilde s_{f^t}}", from=1-1, to=1-2]
	\arrow["\sim", from=2-2, to=3-2]
	\arrow["{\gamma_f \circ_s i}", from=1-2, to=2-2]
\end{tikzcd}\] commute for $f : B_1,...,JX,...,JY,...,B_n \to TZ$.
    \item [(3),\ (4)] Precomposing $\gamma_f$ in the $j$th or $k$th argument with the strengthening of a map $g$ in its $l$th argument: the diagrams
    % https://q.uiver.app/?q=WzAsMTQsWzMsMCwiKGZedFxcY2lyY190IGgpXnt0c30iXSxbMywxLCIoZl50XFxjaXJjX3QgaClee3N0fSJdLFszLDIsIihmXnt0c31cXGNpcmNfdCBoKV50Il0sWzQsMSwiZl57dHN9XFxjaXJjX3QgaF50Il0sWzQsMywiZl57c3R9XFxjaXJjX3QgaF50Il0sWzMsMywiKGZee3N0fVxcY2lyY190IGgpXnQiXSxbNCwwLCIoZl50XFxjaXJjX3QgaF50KV5zIl0sWzAsMywiKGZecyBcXGNpcmNfcyBnKV57c3R9Il0sWzAsMiwiKGZecyBcXGNpcmNfcyBnKV57dHN9Il0sWzEsMywiKGZecyBcXGNpcmNfcyBnXnQpXnQiXSxbMSwyLCJmXntzdH0gXFxjaXJjX3MgZ150Il0sWzEsMCwiZl57dHN9IFxcY2lyY19zIGdedCJdLFswLDEsIihmXntzdH0gXFxjaXJjX3MgZylecyJdLFswLDAsIihmXnt0c30gXFxjaXJjX3MgZylecyJdLFswLDEsIlxcZ2FtbWFfe2ZedFxcY2lyY190IGh9IiwyXSxbMSwyLCJcXHNpbSIsMl0sWzMsNCwiXFxnYW1tYV9mXFxjaXJjX3QgaF50Il0sWzAsNiwiKFxcaGF0IHRfe2YsaH0pXnMiXSxbNiwzLCJcXHNpbSJdLFsyLDUsIihcXGdhbW1hX2ZcXGNpcmNfdCBoKV50IiwyXSxbNSw0LCJcXGhhdCB0X3tmXnMsaH0iLDJdLFs4LDcsIlxcZ2FtbWFfe2ZecyBcXGNpcmNfcyBnfSIsMl0sWzcsOSwiKFxcaGF0IHNfe2YsZ30pXnQiLDJdLFsxMCw5LCJcXHNpbSJdLFsxMSwxMCwiXFxnYW1tYV9mIFxcY2lyY19zIGdedCJdLFsxMiw4LCJcXHNpbSIsMl0sWzEzLDEyLCIoXFxnYW1tYV9mIFxcY2lyY19zIGcpXnMiLDJdLFsxMywxMSwiXFxoYXQgc197Zl50LGd9Il1d
\[\begin{tikzcd}[ampersand replacement=\&]
	{(f^{ts} \circ_s g)^s} \& {f^{ts} \circ_s g^t} \&\& {(f^t\circ_t h)^{ts}} \& {(f^t\circ_t h^t)^s} \\
	{(f^{st} \circ_s g)^s} \&\&\& {(f^t\circ_t h)^{st}} \& {f^{ts}\circ_t h^t} \\
	{(f^s \circ_s g)^{ts}} \& {f^{st} \circ_s g^t} \&\& {(f^{ts}\circ_t h)^t} \\
	{(f^s \circ_s g)^{st}} \& {(f^s \circ_s g^t)^t} \&\& {(f^{st}\circ_t h)^t} \& {f^{st}\circ_t h^t}
	\arrow["{\gamma_{f^t\circ_t h}}"', from=1-4, to=2-4]
	\arrow["\sim"', from=2-4, to=3-4]
	\arrow["{\gamma_f\circ_t h^t}", from=2-5, to=4-5]
	\arrow["{(\hat t_{f,h})^s}", from=1-4, to=1-5]
	\arrow["\sim", from=1-5, to=2-5]
	\arrow["{(\gamma_f\circ_t h)^t}"', from=3-4, to=4-4]
	\arrow["{\hat t_{f^s,h}}"', from=4-4, to=4-5]
	\arrow["{\gamma_{f^s \circ_s g}}"', from=3-1, to=4-1]
	\arrow["{(\hat s_{f,g})^t}"', from=4-1, to=4-2]
	\arrow["\sim", from=3-2, to=4-2]
	\arrow["{\gamma_f \circ_s g^t}", from=1-2, to=3-2]
	\arrow["\sim"', from=2-1, to=3-1]
	\arrow["{(\gamma_f \circ_s g)^s}"', from=1-1, to=2-1]
	\arrow["{\hat s_{f^t,g}}", from=1-1, to=1-2]
\end{tikzcd}\] commute for $f : B_1,...,JX,...,JY,...,B_n \to TZ$, $g : C_1,...,JW,...,C_m \to TX$ and $h : D_1,...,JV,...,D_l \to TY$.
    \item [(5)] Braiding axiom relating the six ways to strengthen a map \[f : B_1,...JW,...,JX,...,JY,...,B_n \to TZ\] in all three arguments: the diagram
    % https://q.uiver.app/?q=WzAsNixbMCwwLCJmXnt1dHN9Il0sWzQsMSwiZl57c3R1fSJdLFsyLDAsImZee3R1c30iXSxbNCwwLCJmXnt0c3V9Il0sWzAsMSwiZl57dXN0fSJdLFsyLDEsImZee3N1dH0iXSxbMCwyLCIoXFxnYW1tYV9mKV5zIl0sWzIsMywiXFxnYW1tYV97Zl50fSJdLFszLDEsIihcXGdhbW1hX2YpXnUiXSxbMCw0LCJcXGdhbW1hX3tmXnV9IiwyXSxbNCw1LCIoXFxnYW1tYV9mKV50IiwyXSxbNSwxLCJcXGdhbW1hX3tmXnN9IiwyXV0=
\[\begin{tikzcd}[ampersand replacement=\&]
	{f^{uts}} \&\& {f^{tus}} \&\& {f^{tsu}} \\
	{f^{ust}} \&\& {f^{sut}} \&\& {f^{stu}}
	\arrow["{(\gamma_f)^s}", from=1-1, to=1-3]
	\arrow["{\gamma_{f^t}}", from=1-3, to=1-5]
	\arrow["{(\gamma_f)^u}", from=1-5, to=2-5]
	\arrow["{\gamma_{f^u}}"', from=1-1, to=2-1]
	\arrow["{(\gamma_f)^t}"', from=2-1, to=2-3]
	\arrow["{\gamma_{f^s}}"', from=2-3, to=2-5]
\end{tikzcd}\] commutes for all $f : B_1,...JW,...,JX,...,JY,...,B_n \to TZ$.
    \end{enumerate}
\end{definition}

\begin{remark}
    When $J$ is the identity, this definition reduces to the definition of pseudocommutativity found in \cite{hp2002} Definition 5. The correspondence between the coherence conditions given here and their conditions is enumerated in the following table:
    \begin{table}[ht]
        \centering
        \begin{tabular}{c|c}
        Relative setting & Hyland \& Power \\
        (1), (2) & 4., 5. \\
        (3), (4) & 6., 7. \\
        (5) & 1., 2., 3.
    \end{tabular}
    \end{table}
\end{remark}

\begin{remark}
    The braiding axiom (5) allows us to extend our notation. Given a map $f : JX_1,...,JX_n \to TY$ and a permutation $\sigma \in S_n$, we can construct maps
    \[f^{t_1...t_n} \to f^{t_{\sigma(1)}...t_{\sigma(n)}}\] as a composite of $\gamma$ maps and their inverses. The braiding axiom (5) tells us that any two such composites of $\gamma$ and $\gamma^{-1}$ maps are equal; we will denote this map by
    \[\gamma_{\sigma;f} : f^{t_1...t_n} \to f^{t_{\sigma(1)}...t_{\sigma(n)}}.\]
\end{remark}

\begin{example}
    The presheaf relative pseudomonad will turn out to be pseudocommutative in this sense; recalling the formula for the multicategorical action of $\Psh$ on 1-cells in Example~\ref{psh is multi psfun}, one should be able to permute the order of strengthenings by means of Fubini isomorphisms for coends. However, proving that $\Psh$ is pseudocommutative directly in this way is challenging; in section 5 we will discuss a property that implies pseudocommutativity and which is much easier to verify.
\end{example}

In Kock \cite{kock1970} we have that a strong monad is lax-monoidal as a functor, and even that the monad unit for a strong monad is a monoidal transformation, but that in order for the monad multiplication (and thus the monad as a whole) to be monoidal, the monad must be commutative. In our setting, every strong relative pseudomonad $T$ has the structure of a multicategorical \emph{pseudofunctor}, and we are now interested in the question of when $T$ further has the structure of a multicategorical \emph{relative pseudomonad} (defined below); that is, when the pseudomonadic structure of $T$ is compatible with the ambient multicategorical structure. We will show in this section that every pseudocommutative relative pseudomonad is a multicategorical relative pseudomonad.

\begin{definition}
    (Multicategorical relative pseudomonad) Let $\C,\D$ be a pair of 2-multicategories and let $T$ be a relative pseudomonad along $J : \D \to \C$. We say $T$ is a \emph{multicategorical relative pseudomonad} if
    \begin{itemize}
        \item $T$ is a pseudo-multifunctor, and
        \item The unit and extension of $T$ are compatible with the multicategorical structure.
    \end{itemize}
    For the second bullet point, we explicitly ask that
    \begin{itemize}
        \item the monad unit $i$ is multicategorical: for each $f : X_1,...,X_n \to Y$ we have an invertible 2-cell \[\bar \imath_f : i_Y \circ Jf \to Tf \circ (i_{X_1},...,i_{X_n}),\]
        % https://q.uiver.app/?q=WzAsNCxbMCwwLCJKWF8xLC4uLixKWF9uIl0sWzAsMSwiSlkiXSxbMSwwLCJUWF8xLC4uLixUWF9uIl0sWzEsMSwiVFkiXSxbMCwxLCJKZiIsMl0sWzIsMywiVGYiXSxbMSwzLCJpIiwyXSxbMCwyLCJpLC4uLixpIl0sWzEsMiwiXFxiYXIgXFxpbWF0aF9mIiwwLHsic2hvcnRlbiI6eyJzb3VyY2UiOjIwLCJ0YXJnZXQiOjIwfSwibGV2ZWwiOjJ9XV0=
\[\begin{tikzcd}[ampersand replacement=\&]
	{JX_1,...,JX_n} \& {TX_1,...,TX_n} \\
	JY \& TY
	\arrow["Jf"', from=1-1, to=2-1]
	\arrow["Tf", from=1-2, to=2-2]
	\arrow["i"', from=2-1, to=2-2]
	\arrow["{i,...,i}", from=1-1, to=1-2]
	\arrow["{\bar \imath_f}", shorten <=7pt, shorten >=7pt, Rightarrow, from=2-1, to=1-2]
\end{tikzcd}\]
        \item the monad extension $(-)^*$ is multicategorical: for each 2-cell of the form $\alpha : h \circ Jf \to Tf' \circ (g_1,...,g_n) $:
        % https://q.uiver.app/?q=WzAsNCxbMCwwLCJKWF8xLC4uLixKWF9uIl0sWzAsMSwiSlkiXSxbMiwwLCJUWF8xJywuLi4sVFhfbiciXSxbMiwxLCJUWSciXSxbMCwxLCJKZiIsMl0sWzIsMywiVGYnIl0sWzEsMywiaCIsMl0sWzAsMiwiZ18xLC4uLixnX24iXSxbMSwyLCJcXGFscGhhIiwwLHsic2hvcnRlbiI6eyJzb3VyY2UiOjIwLCJ0YXJnZXQiOjIwfSwibGV2ZWwiOjJ9XV0=
\[\begin{tikzcd}[ampersand replacement=\&]
	{JX_1,...,JX_n} \&\& {TX_1',...,TX_n'} \\
	JY \&\& {TY'}
	\arrow["Jf"', from=1-1, to=2-1]
	\arrow["{Tf'}", from=1-3, to=2-3]
	\arrow["h"', from=2-1, to=2-3]
	\arrow["{g_1,...,g_n}", from=1-1, to=1-3]
	\arrow["\alpha", shorten <=12pt, shorten >=12pt, Rightarrow, from=2-1, to=1-3]
\end{tikzcd}\] we have a 2-cell $\alpha^* :h^* \circ Tf \to  Tf' \circ (g_1^*,...,g_n^*)$ fitting into the square
% https://q.uiver.app/?q=WzAsNCxbMCwwLCJUWF8xLC4uLixUWF9uIl0sWzAsMSwiVFkiXSxbMiwwLCJUWF8xJywuLi4sVFhfbiciXSxbMiwxLCJUWSciXSxbMCwxLCJUZiIsMl0sWzIsMywiVGYnIl0sWzEsMywiaF4qIiwyXSxbMCwyLCJnXzFeKiwuLi4sZ19uXioiXSxbMSwyLCJcXGFscGhhXioiLDAseyJzaG9ydGVuIjp7InNvdXJjZSI6MjAsInRhcmdldCI6MjB9LCJsZXZlbCI6Mn1dXQ==
\[\begin{tikzcd}[ampersand replacement=\&]
	{TX_1,...,TX_n} \&\& {TX_1',...,TX_n'} \\
	TY \&\& {TY'}
	\arrow["Tf"', from=1-1, to=2-1]
	\arrow["{Tf'}", from=1-3, to=2-3]
	\arrow["{h^*}"', from=2-1, to=2-3]
	\arrow["{g_1^*,...,g_n^*}", from=1-1, to=1-3]
	\arrow["{\alpha^*}", shorten <=12pt, shorten >=12pt, Rightarrow, from=2-1, to=1-3]
\end{tikzcd}\]
    \end{itemize}
    These must satisfy three coherence conditions (one for each of the families of 2-cells making $T$ a relative pseudomonad).
    \begin{enumerate}
        \item[(1)] Compatibility with $\eta$: given a 2-cell $\alpha : h \circ Jf \to Tf' \circ (g_1,...,g_n) $, the composite
        % https://q.uiver.app/?q=WzAsNixbMiwwLCJUWF8xLC4uLixUWF9uIl0sWzIsMSwiVFkiXSxbNCwwLCJUWF8xJywuLi4sVFhfbiciXSxbNCwxLCJUWSciXSxbMCwwLCJKWF8xLC4uLixKWF9uIl0sWzAsMSwiSlkiXSxbMCwxLCJUZiIsMl0sWzIsMywiVGYnIl0sWzEsMywiaF4qIiwyXSxbMCwyLCJnXzFeKiwuLi4sZ19uXioiXSxbMSwyLCJcXGFscGhhXioiLDAseyJzaG9ydGVuIjp7InNvdXJjZSI6MjAsInRhcmdldCI6MjB9LCJsZXZlbCI6Mn1dLFs0LDAsImksLi4uLGkiXSxbNSwxLCJpIiwyXSxbNCw1LCJKZiIsMl0sWzUsMCwiXFxiYXIgXFxpbWF0aF9mIiwwLHsic2hvcnRlbiI6eyJzb3VyY2UiOjIwLCJ0YXJnZXQiOjIwfSwibGV2ZWwiOjJ9XSxbNSwzLCJoIiwyLHsiY3VydmUiOjV9XSxbMTUsMSwiXFxldGFfaCIsMCx7InNob3J0ZW4iOnsic291cmNlIjoyMCwidGFyZ2V0IjoyMH19XV0=
\[\begin{tikzcd}[ampersand replacement=\&]
	{JX_1,...,JX_n} \&\& {TX_1,...,TX_n} \&\& {TX_1',...,TX_n'} \\
	JY \&\& TY \&\& {TY'}
	\arrow["Tf"', from=1-3, to=2-3]
	\arrow["{Tf'}", from=1-5, to=2-5]
	\arrow["{h^*}"', from=2-3, to=2-5]
	\arrow["{g_1^*,...,g_n^*}", from=1-3, to=1-5]
	\arrow["{\alpha^*}", shorten <=12pt, shorten >=12pt, Rightarrow, from=2-3, to=1-5]
	\arrow["{i,...,i}", from=1-1, to=1-3]
	\arrow["i"', from=2-1, to=2-3]
	\arrow["Jf"', from=1-1, to=2-1]
	\arrow["{\bar \imath_f}", shorten <=12pt, shorten >=12pt, Rightarrow, from=2-1, to=1-3]
	\arrow[""{name=0, anchor=center, inner sep=0}, "h"', curve={height=30pt}, from=2-1, to=2-5]
	\arrow["{\eta_h}", shorten <=3pt, shorten >=3pt, Rightarrow, from=0, to=2-3]
\end{tikzcd}\] is equal to the composite
% https://q.uiver.app/?q=WzAsNSxbMCwxLCJKWF8xLC4uLixKWF9uIl0sWzIsMSwiVFhfMScsLi4uLFRYX24nIl0sWzAsMiwiSlkiXSxbMiwyLCJUWSciXSxbMSwwLCJUWF8xLC4uLixUWF9uIl0sWzAsMSwiZ18xLC4uLixnX24iLDFdLFswLDIsIkpmIiwyXSxbMSwzLCJUZiciXSxbMiwzLCJoIiwyXSxbMiwxLCJcXGFscGhhIiwwLHsic2hvcnRlbiI6eyJzb3VyY2UiOjIwLCJ0YXJnZXQiOjIwfSwibGV2ZWwiOjJ9XSxbMCw0LCJpLC4uLixpIiwwLHsiY3VydmUiOi0yfV0sWzQsMSwiZ18xXiosLi4uLGdfbl4qIiwwLHsiY3VydmUiOi0yfV0sWzUsNCwiXFxldGFfe2dfMX0sLi4uLFxcZXRhX3tnX259IiwwLHsic2hvcnRlbiI6eyJzb3VyY2UiOjIwfX1dXQ==
\[\begin{tikzcd}[ampersand replacement=\&]
	\& {TX_1,...,TX_n} \\
	{JX_1,...,JX_n} \&\& {TX_1',...,TX_n'} \\
	JY \&\& {TY'}
	\arrow[""{name=0, anchor=center, inner sep=0}, "{g_1,...,g_n}"{description}, from=2-1, to=2-3]
	\arrow["Jf"', from=2-1, to=3-1]
	\arrow["{Tf'}", from=2-3, to=3-3]
	\arrow["h"', from=3-1, to=3-3]
	\arrow["\alpha", shorten <=18pt, shorten >=18pt, Rightarrow, from=3-1, to=2-3]
	\arrow["{i,...,i}", curve={height=-12pt}, from=2-1, to=1-2]
	\arrow["{g_1^*,...,g_n^*}", curve={height=-12pt}, from=1-2, to=2-3]
	\arrow["{\eta_{g_1},...,\eta_{g_n}}", shorten <=3pt, Rightarrow, from=0, to=1-2]
\end{tikzcd}\]
    \item[(2)] Compatibility with $\mu$: given 2-cells $\alpha : Tf' \circ (g_1^*,...,g_n^*) \to h \circ Jf$ and $\beta : Tf'' \circ (g_1'^*,...,g_n'^*) \to h' \circ Jf'$, the composite
    % https://q.uiver.app/?q=WzAsNixbMCwwLCJUWF8xLC4uLixUWF9uIl0sWzIsMCwiVFhfMScsLi4uLFRYX24nIl0sWzAsMSwiVFkiXSxbMiwxLCJUWSciXSxbNCwwLCJUWF8xJycsLi4uLFRYX24nJyJdLFs0LDEsIlRZJyciXSxbMCwxLCJnXzFeKiwuLi4sZ19uXioiXSxbMCwyLCJUZiIsMl0sWzEsMywiVGYnIl0sWzIsMywiaF4qIiwyXSxbMiwxLCJcXGFscGhhXioiLDAseyJzaG9ydGVuIjp7InNvdXJjZSI6MjAsInRhcmdldCI6MjB9LCJsZXZlbCI6Mn1dLFsxLDQsImdfMSdeKiwuLi4sZ19uJ14qIl0sWzQsNSwiVGYnJyJdLFszLDUsImgnXioiLDJdLFszLDQsIlxcYmV0YV4qIiwwLHsic2hvcnRlbiI6eyJzb3VyY2UiOjIwLCJ0YXJnZXQiOjIwfSwibGV2ZWwiOjJ9XSxbMiw1LCIoaCdeKmgpXioiLDIseyJjdXJ2ZSI6NX1dLFsxNSwzLCJcXG11X3toJyxofSIsMCx7InNob3J0ZW4iOnsic291cmNlIjoyMH19XV0=
\[\begin{tikzcd}[ampersand replacement=\&]
	{TX_1,...,TX_n} \&\& {TX_1',...,TX_n'} \&\& {TX_1'',...,TX_n''} \\
	TY \&\& {TY'} \&\& {TY''}
	\arrow["{g_1^*,...,g_n^*}", from=1-1, to=1-3]
	\arrow["Tf"', from=1-1, to=2-1]
	\arrow["{Tf'}", from=1-3, to=2-3]
	\arrow["{h^*}"', from=2-1, to=2-3]
	\arrow["{\alpha^*}", shorten <=12pt, shorten >=12pt, Rightarrow, from=2-1, to=1-3]
	\arrow["{g_1'^*,...,g_n'^*}", from=1-3, to=1-5]
	\arrow["{Tf''}", from=1-5, to=2-5]
	\arrow["{h'^*}"', from=2-3, to=2-5]
	\arrow["{\beta^*}", shorten <=12pt, shorten >=12pt, Rightarrow, from=2-3, to=1-5]
	\arrow[""{name=0, anchor=center, inner sep=0}, "{(h'^*h)^*}"', curve={height=30pt}, from=2-1, to=2-5]
	\arrow["{\mu_{h',h}}", shorten <=3pt, Rightarrow, from=0, to=2-3]
\end{tikzcd}\] is equal to the composite
% https://q.uiver.app/?q=WzAsNSxbMCwxLCJUWF8xLC4uLixUWF9uIl0sWzAsMiwiVFkiXSxbNCwxLCJUWF8xJycsLi4uLFRYX24nJyJdLFs0LDIsIlRZJyciXSxbMiwwLCJUWF8xJywuLi4sVFhfbiciXSxbMCwxLCJUZiIsMl0sWzIsMywiVGYnJyJdLFswLDIsIihnXzEnXipnXzEpXiosLi4uLChnX24nXipnX24pXioiLDFdLFsxLDMsIihoJ14qaCleKiIsMl0sWzEsMiwiKFxcYmV0YV4qXFxhbHBoYSleKiIsMCx7InNob3J0ZW4iOnsic291cmNlIjoyMCwidGFyZ2V0IjoyMH0sImxldmVsIjoyfV0sWzAsNCwiZ18xXiosLi4uLGdfbl4qIiwwLHsiY3VydmUiOi0yfV0sWzQsMiwiZ18xJ14qLC4uLixnX24nXioiLDAseyJjdXJ2ZSI6LTJ9XSxbNyw0LCJcXG11X3tnXzEnLGdfMX0sLi4uLFxcbXVfe2dfbicsZ19ufSIsMCx7InNob3J0ZW4iOnsic291cmNlIjoyMH19XV0=
\[\begin{tikzcd}[ampersand replacement=\&]
	\&\& {TX_1',...,TX_n'} \\
	{TX_1,...,TX_n} \&\&\&\& {TX_1'',...,TX_n''} \\
	TY \&\&\&\& {TY''}
	\arrow["Tf"', from=2-1, to=3-1]
	\arrow["{Tf''}", from=2-5, to=3-5]
	\arrow[""{name=0, anchor=center, inner sep=0}, "{(g_1'^*g_1)^*,...,(g_n'^*g_n)^*}"{description}, from=2-1, to=2-5]
	\arrow["{(h'^*h)^*}"', from=3-1, to=3-5]
	\arrow["{(\beta^*\alpha)^*}", shorten <=31pt, shorten >=31pt, Rightarrow, from=3-1, to=2-5]
	\arrow["{g_1^*,...,g_n^*}", curve={height=-12pt}, from=2-1, to=1-3]
	\arrow["{g_1'^*,...,g_n'^*}", curve={height=-12pt}, from=1-3, to=2-5]
	\arrow["{\mu_{g_1',g_1},...,\mu_{g_n',g_n}}", shorten <=3pt, Rightarrow, from=0, to=1-3]
\end{tikzcd}\] (where for clarity we have omitted whiskerings from the 2-cell here written $(\beta^*\alpha)^*$), and
    \item[(3)] Compatibility with $\theta$: given $f : X_1,...,X_n \to Y$, the composite
    % https://q.uiver.app/?q=WzAsNCxbMCwwLCJUWF8xLC4uLixUWF9uIl0sWzIsMCwiVFhfMSwuLi4sVFhfbiJdLFswLDEsIlRZIl0sWzIsMSwiVFkiXSxbMSwzLCJUZiJdLFswLDIsIlRmIiwyXSxbMiwzLCJpXioiLDJdLFswLDEsImleKiwuLi4saV4qIiwxXSxbMiwxLCJcXGJhciBcXGltYXRoX2ZeKiIsMCx7InNob3J0ZW4iOnsic291cmNlIjoyMCwidGFyZ2V0IjoyMH0sImxldmVsIjoyfV0sWzAsMSwiMSwuLi4sMSIsMCx7ImN1cnZlIjotNX1dLFs3LDksIlxcdGhldGEsLi4uLFxcdGhldGEiLDAseyJzaG9ydGVuIjp7InNvdXJjZSI6MjAsInRhcmdldCI6MjB9fV1d
\[\begin{tikzcd}[ampersand replacement=\&]
	{TX_1,...,TX_n} \&\& {TX_1,...,TX_n} \\
	TY \&\& TY
	\arrow["Tf", from=1-3, to=2-3]
	\arrow["Tf"', from=1-1, to=2-1]
	\arrow["{i^*}"', from=2-1, to=2-3]
	\arrow[""{name=0, anchor=center, inner sep=0}, "{i^*,...,i^*}"{description}, from=1-1, to=1-3]
	\arrow["{\bar \imath_f^*}", shorten <=12pt, shorten >=12pt, Rightarrow, from=2-1, to=1-3]
	\arrow[""{name=1, anchor=center, inner sep=0}, "{1,...,1}", curve={height=-30pt}, from=1-1, to=1-3]
	\arrow["{\theta,...,\theta}", shorten <=4pt, shorten >=4pt, Rightarrow, from=0, to=1]
\end{tikzcd}\] is equal to the composite
% https://q.uiver.app/?q=WzAsNCxbMCwwLCJUWF8xLC4uLixUWF9uIl0sWzIsMCwiVFhfMSwuLi4sVFhfbiJdLFswLDEsIlRZIl0sWzIsMSwiVFkiXSxbMSwzLCJUZiJdLFswLDIsIlRmIiwyXSxbMiwzLCIxIl0sWzAsMSwiMSwuLi4sMSJdLFsyLDMsImleKiIsMix7ImN1cnZlIjo1fV0sWzgsNiwiXFx0aGV0YSIsMix7InNob3J0ZW4iOnsic291cmNlIjoyMCwidGFyZ2V0IjoyMH19XV0=
\[\begin{tikzcd}[ampersand replacement=\&]
	{TX_1,...,TX_n} \&\& {TX_1,...,TX_n} \\
	TY \&\& TY
	\arrow["Tf", from=1-3, to=2-3]
	\arrow["Tf"', from=1-1, to=2-1]
	\arrow[""{name=0, anchor=center, inner sep=0}, "1", from=2-1, to=2-3]
	\arrow["{1,...,1}", from=1-1, to=1-3]
	\arrow[""{name=1, anchor=center, inner sep=0}, "{i^*}"', curve={height=30pt}, from=2-1, to=2-3]
	\arrow["\theta"', shorten <=4pt, shorten >=4pt, Rightarrow, from=1, to=0]
\end{tikzcd}\]
    \end{enumerate}
\end{definition}

\begin{remark}
    In the one-dimensional, monoidal setting and when $J$ is the identity, this definition reduces to the notion of a monoidal monad.
\end{remark}

In \cite{kock1970} it is noted that the monad unit of a strong monad is always a monoidal transformation, but the monad multiplication is only a monoidal transformation if the monad is commutative. We shall see in the following proposition an analogous result: that for every strong relative pseudomonad, the monad unit is multicategorical (we can define the invertible 2-cells $\bar \imath_f$), but in order to make the monad extension multicategorical we require the relative pseudomonad to be pseudocommutative. 

\begin{theorem}\label{pscom=>multmonad}
    Let $T$ be a strong relative pseudomonad along multicategorical 2-functor $J : \D \to \C$. Suppose $T$ is pseudocommutative. Then $T$ is a multicategorical relative pseudomonad.
\end{theorem}
\begin{proof}
    By Proposition~\ref{param=>psfunctor} we know that $T$ is a pseudo-multifunctor. We must check that the monad unit and extension are compatible with the multicategorical structure. For the unit, we need to find invertible 2-cells $\bar \imath_f$ of shape
    \[i \circ Jf \to Tf \circ (i,...,i)\] for $f : X_1,...,X_n \to Y$. Since $Tf := (i \circ Jf)^{t_1...t_n} = \bar f^{t_1...t_n}$, we construct $\bar \imath_f$ as the composite
    \begin{align*}
        i \circ Jf = \bar f &\xrightarrow{\tilde t} \bar f^{t_1} \circ (i,1,...,1)\\
        &\xrightarrow{\tilde t} \bar f^{t_1t_2} \circ (i,i,1,...,1)\\
        &\vdots\\
        &\xrightarrow{\tilde t} \bar f^{t_1t_2...t_n} \circ (i,i,i,...,i) = Tf \circ (i,...,i).
    \end{align*}
    Note that we do not need the pseudocommutativity to construct the $\bar \imath_f$ 2-cells. The construction of $\alpha^*$ given $\alpha : h \circ Jf \to Tf' \circ (g_1,...,g_n)$ is more involved. We require a 2-cell of shape
    \[h^* \circ Tf \to Tf' \circ (g_1^*,...,g_n^*).\] We begin with the composite
    \begin{align*}
        h^* \circ Tf := h^t \circ \bar f^{t_1...t_n} &\xrightarrow{\hat t^{-1}} (h^t \circ \bar f^{t_1...t_{n-1}})^{t_n}\\
        &\xrightarrow{\hat t^{-1}} (h^t \circ \bar f^{t_1...t_{n-2}})^{t_{n-1}t_n}\\
        &\vdots\\
        &\xrightarrow{\hat t^{-1}} (h^t \circ \bar f)^{t_1...t_n}\\
        &\xrightarrow{\tilde t^{-1}} (h \circ Jf)^{t_1...t_n},
    \end{align*} at which point we can compose with $\alpha^{t_1...t_n}$ to arrive at
    \[(Tf' \circ (g_1,...,g_n))^{t_1...t_n} := (\bar f'^{t_1....t_n} \circ (g_1,...,g_n))^{t_1...t_n}.\] From here we start needing the pseudocommutativity of $T$. Let $\sigma \in S_n$ be the cyclic permutation $1 \to 2 \to ... \to n \to 1$. Now we compose as follows:
    \begin{align*}
        &(\bar f'^{t_1....t_n} \circ (g_1,...,g_n))^{t_1...t_n}\\
        &\xrightarrow{\gamma_\sigma} (\bar f'^{t_2....t_1} \circ (g_1,...,g_n))^{t_1...t_n} \xrightarrow{\hat t} (\bar f'^{t_2....t_1} \circ (g_1^t,g_2,...,g_n))^{t_2...t_n}\\
        &\xrightarrow{\gamma_\sigma} (\bar f'^{t_3....t_2} \circ (g_1^t,g_2,...,g_n))^{t_2...t_n} \xrightarrow{\hat t} (\bar f'^{t_3....t_2} \circ (g_1^t,g_2^t,g_3,...,g_n))^{t_3...t_n}\\
        &\vdots\\
        &\xrightarrow{\gamma_\sigma} (\bar f'^{t_1....t_n} \circ (g_1^t,...,g_{n-1}^t,g_t))^{t_n} \xrightarrow{\hat t} \bar f'^{t_1....t_n} \circ (g_1^t,...,,g_n^t)\\
        &= Tf' \circ (g_1^*,...,g_n^*).
    \end{align*}
    For example, the full composite in the case where $f$ is a binary map is given by the diagram below:
    % https://q.uiver.app/?q=WzAsOSxbMCwwLCJoXnQgXFxjaXJjIFxcYmFyIGZee3N0fSJdLFswLDIsIihoXnQgXFxjaXJjIFxcYmFyIGZecyledCJdLFswLDQsIihoXnQgXFxjaXJjIFxcYmFyIGYpXntzdH0iXSxbMSw0LCIoaCBcXGNpcmMgSmYpXntzdH0iXSxbMiw0LCIoZidee3N0fSBcXGNpcmMgKGdfMSxnXzIpKV57c3R9Il0sWzIsMywiKGYnXnt0c30gXFxjaXJjIChnXzEsZ18yKSlee3N0fSJdLFsyLDIsIihmJ157dHN9IFxcY2lyYyAoZ18xXiosZ18yKSlee3R9Il0sWzIsMSwiKGYnXntzdH0gXFxjaXJjIChnXzFeKixnXzIpKV57dH0iXSxbMiwwLCJmJ157c3R9IFxcY2lyYyAoZ18xXiosZ18yXiopIl0sWzAsMSwiXFxoYXQgdF57LTF9IiwyXSxbMSwyLCJcXGhhdCBzXnstMX0iLDJdLFsyLDMsIlxcdGlsZGUgdF57LTF9IiwyXSxbMyw0LCJcXGFscGhhIiwyXSxbNCw1LCJcXGdhbW1hXnstMX0iLDJdLFs1LDYsIlxcaGF0IHMiLDJdLFs2LDcsIlxcZ2FtbWEiLDJdLFs3LDgsIlxcaGF0IHQiLDJdLFswLDgsIlxcYWxwaGFeKiIsMCx7InN0eWxlIjp7ImJvZHkiOnsibmFtZSI6ImRhc2hlZCJ9fX1dXQ==
\[\begin{tikzcd}[ampersand replacement=\&]
	{h^t \circ \bar f^{st}} \&\& {f'^{st} \circ (g_1^*,g_2^*)} \\
	\&\& {(f'^{st} \circ (g_1^*,g_2))^{t}} \\
	{(h^t \circ \bar f^s)^t} \&\& {(f'^{ts} \circ (g_1^*,g_2))^{t}} \\
	\&\& {(f'^{ts} \circ (g_1,g_2))^{st}} \\
	{(h^t \circ \bar f)^{st}} \& {(h \circ Jf)^{st}} \& {(f'^{st} \circ (g_1,g_2))^{st}}
	\arrow["{\hat t^{-1}}"', from=1-1, to=3-1]
	\arrow["{\hat s^{-1}}"', from=3-1, to=5-1]
	\arrow["{\tilde t^{-1}}"', from=5-1, to=5-2]
	\arrow["\alpha"', from=5-2, to=5-3]
	\arrow["{\gamma^{-1}}"', from=5-3, to=4-3]
	\arrow["{\hat s}"', from=4-3, to=3-3]
	\arrow["\gamma"', from=3-3, to=2-3]
	\arrow["{\hat t}"', from=2-3, to=1-3]
	\arrow["{\alpha^*}", dashed, from=1-1, to=1-3]
\end{tikzcd}\]

It now remains to verify that the three coherence conditions for a multicategorical relative pseudomonad. Here we shall only do this for binary maps, and we shall abbreviate the diagram chasing.

For the first condition, we begin with the composite 2-cell
% https://q.uiver.app/?q=WzAsNCxbMCwwLCJoIFxcY2lyYyBKZiJdLFsxLDAsImheKiBcXGNpcmMgaSBcXGNpcmMgSmYiXSxbMiwwLCJoXiogXFxjaXJjIFRmIFxcY2lyYyAoaSxpKSJdLFszLDAsIlRmJyBcXGNpcmMgKGdfMV4qLGdfMl4qKSBcXGNpcmMgKGksaSkiXSxbMCwxLCJcXGV0YSJdLFsxLDIsIlxcYmFyIFxcaW1hdGgiXSxbMiwzLCJcXGFscGhhXioiXV0=
\[\begin{tikzcd}[ampersand replacement=\&]
	{h \circ Jf} \& {h^* \circ i \circ Jf} \& {h^* \circ Tf \circ (i,i)} \& {Tf' \circ (g_1^*,g_2^*) \circ (i,i)}
	\arrow["\eta", from=1-1, to=1-2]
	\arrow["{\bar \imath}", from=1-2, to=1-3]
	\arrow["{\alpha^*}", from=1-3, to=1-4]
\end{tikzcd}\] and must show that it is equal to the composite
% https://q.uiver.app/?q=WzAsMyxbMCwwLCJoIFxcY2lyYyBKZiJdLFsyLDAsIlRmJyBcXGNpcmMgKGdfMV4qLGdfMl4qKSBcXGNpcmMgKGksaSkiXSxbMSwwLCJUZicgXFxjaXJjIChnXzEsZ18yKSJdLFswLDIsIlxcYWxwaGEiXSxbMiwxLCJcXGV0YSxcXGV0YSJdXQ==
\[\begin{tikzcd}[ampersand replacement=\&]
	{h \circ Jf} \& {Tf' \circ (g_1,g_2)} \& {Tf' \circ (g_1^*,g_2^*) \circ (i,i).}
	\arrow["\alpha", from=1-1, to=1-2]
	\arrow["{\eta,\eta}", from=1-2, to=1-3]
\end{tikzcd}\] Rewriting $\bar \imath_f$ and $\alpha^*$ in terms of our constructions, we must show that the diagram
\begin{small}
    % https://q.uiver.app/?q=WzAsMTQsWzAsMCwiaCBcXGNpcmMgSmYiXSxbMSwwLCJoXnQgXFxjaXJjIFxcYmFyIGYiXSxbMywwLCJoXnQgXFxjaXJjIFxcYmFyIGZee3N0fSBcXGNpcmMgKGksaSkiXSxbMyw4LCJcXGJhciBmJ157c3R9IFxcY2lyYyAoZ18xXnQsZ18yXnQpIFxcY2lyYyAoaSxpKSJdLFswLDgsIlxcYmFyIGYnXntzdH0gXFxjaXJjIChnXzEsZ18yKSJdLFsyLDAsImhedCBcXGNpcmMgXFxiYXIgZl5zIFxcY2lyYyAoaSwxKSJdLFszLDEsIihoXnQgXFxjaXJjIFxcYmFyIGZee3N9KV50IFxcY2lyYyAoaSxpKSJdLFszLDIsIihoXnQgXFxjaXJjIFxcYmFyIGYpXntzdH0gXFxjaXJjIChpLGkpIl0sWzMsMywiKGggXFxjaXJjIEpmKV57c3R9IFxcY2lyYyAoaSxpKSJdLFszLDUsIihcXGJhciBmJ157dHN9IFxcY2lyYyAoZ18xLGdfMikpXntzdH0gXFxjaXJjIChpLGkpIl0sWzMsNCwiKFxcYmFyIGYnXntzdH0gXFxjaXJjIChnXzEsZ18yKSlee3N0fSBcXGNpcmMgKGksaSkiXSxbMyw2LCIoXFxiYXIgZidee3RzfSBcXGNpcmMgKGdfMV50LGdfMikpXnQgXFxjaXJjIChpLGkpIl0sWzMsNywiKFxcYmFyIGYnXntzdH0gXFxjaXJjIChnXzFedCxnXzIpKV50IFxcY2lyYyAoaSxpKSJdLFsyLDgsIlxcYmFyIGYnXntzdH0gXFxjaXJjIChnXzFedCxnXzIpIFxcY2lyYyAoaSwxKSJdLFswLDEsIlxcdGlsZGUgdCJdLFswLDQsIlxcYWxwaGEiLDJdLFsxLDUsIlxcdGlsZGUgcyJdLFs1LDIsIlxcdGlsZGUgdCJdLFsyLDYsIlxcaGF0IHReey0xfSJdLFs2LDcsIlxcaGF0IHNeey0xfSJdLFs3LDgsIlxcdGlsZGUgdF57LTF9Il0sWzEwLDksIlxcZ2FtbWFeey0xfSJdLFs4LDEwLCJcXGFscGhhIl0sWzksMTEsIlxcaGF0IHMiXSxbMTEsMTIsIlxcZ2FtbWEiXSxbMTIsMywiXFxoYXQgdCJdLFs0LDEzLCJcXHRpbGRlIHQiLDJdLFsxMywzLCJcXHRpbGRlIHQiLDJdXQ==
\[\begin{tikzcd}[ampersand replacement=\&]
	{h \circ Jf} \& {h^t \circ \bar f} \& {h^t \circ \bar f^s \circ (i,1)} \& {h^t \circ \bar f^{st} \circ (i,i)} \\
	\&\&\& {(h^t \circ \bar f^{s})^t \circ (i,i)} \\
	\&\&\& {(h^t \circ \bar f)^{st} \circ (i,i)} \\
	\&\&\& {(h \circ Jf)^{st} \circ (i,i)} \\
	\&\&\& {(\bar f'^{st} \circ (g_1,g_2))^{st} \circ (i,i)} \\
	\&\&\& {(\bar f'^{ts} \circ (g_1,g_2))^{st} \circ (i,i)} \\
	\&\&\& {(\bar f'^{ts} \circ (g_1^t,g_2))^t \circ (i,i)} \\
	\&\&\& {(\bar f'^{st} \circ (g_1^t,g_2))^t \circ (i,i)} \\
	{\bar f'^{st} \circ (g_1,g_2)} \&\& {\bar f'^{st} \circ (g_1^t,g_2) \circ (i,1)} \& {\bar f'^{st} \circ (g_1^t,g_2^t) \circ (i,i)}
	\arrow["{\tilde t}", from=1-1, to=1-2]
	\arrow["\alpha"', from=1-1, to=9-1]
	\arrow["{\tilde s}", from=1-2, to=1-3]
	\arrow["{\tilde t}", from=1-3, to=1-4]
	\arrow["{\hat t^{-1}}", from=1-4, to=2-4]
	\arrow["{\hat s^{-1}}", from=2-4, to=3-4]
	\arrow["{\tilde t^{-1}}", from=3-4, to=4-4]
	\arrow["{\gamma^{-1}}", from=5-4, to=6-4]
	\arrow["\alpha", from=4-4, to=5-4]
	\arrow["{\hat s}", from=6-4, to=7-4]
	\arrow["\gamma", from=7-4, to=8-4]
	\arrow["{\hat t}", from=8-4, to=9-4]
	\arrow["{\tilde t}"', from=9-1, to=9-3]
	\arrow["{\tilde t}"', from=9-3, to=9-4]
\end{tikzcd}\]
\end{small}  commutes. We can fill in this diagram, aside from naturality squares, with four instances of equality (1) from Lemma~\ref{param3moreeqs}. So indeed the first coherence condition holds.

For the second coherence condition, we begin with the composite 2-cell
% https://q.uiver.app/?q=WzAsNCxbMCwwLCIoaCdeKiBcXGNpcmMgaCleKiBcXGNpcmMgVGYiXSxbMSwwLCJoJ14qIFxcY2lyYyBoXiogXFxjaXJjIFRmIl0sWzIsMCwiaCdeKiBcXGNpcmMgVGYnIFxcY2lyYyAoZ18xXiosZ18yXiopIl0sWzIsMSwiVGYnJyBcXGNpcmMgKGdfMSdeKixnXzInXiopIFxcY2lyYyAoZ18xXiosZ18yXiopIl0sWzAsMSwiXFxtdSJdLFsxLDIsIlxcYWxwaGFeKiJdLFsyLDMsIlxcYmV0YV4qIl1d
\[\begin{tikzcd}[ampersand replacement=\&]
	{(h'^* \circ h)^* \circ Tf} \& {h'^* \circ h^* \circ Tf} \& {h'^* \circ Tf' \circ (g_1^*,g_2^*)} \\
	\&\& {Tf'' \circ (g_1'^*,g_2'^*) \circ (g_1^*,g_2^*)}
	\arrow["\mu", from=1-1, to=1-2]
	\arrow["{\alpha^*}", from=1-2, to=1-3]
	\arrow["{\beta^*}", from=1-3, to=2-3]
\end{tikzcd}\] and must show that it is equal to the composite
% https://q.uiver.app/?q=WzAsMyxbMCwwLCIoaCdeKiBcXGNpcmMgaCleKiBcXGNpcmMgVGYiXSxbMiwwLCJUZicnIFxcY2lyYyAoZ18xJ14qLGdfMideKikgXFxjaXJjIChnXzFeKixnXzJeKikiXSxbMSwwLCJUZicnIFxcY2lyYyAoKGdfMSdeKmdfMSleKiwoZ18yJ14qZ18yKV4qKSJdLFswLDIsIihcXGJldGFeKlxcYWxwaGEpXioiXSxbMiwxLCJcXG11LFxcbXUiXV0=
\[\begin{tikzcd}[ampersand replacement=\&]
	{(h'^* \circ h)^* \circ Tf} \& {Tf'' \circ ((g_1'^*g_1)^*,(g_2'^*g_2)^*)} \& {Tf'' \circ (g_1'^*,g_2'^*) \circ (g_1^*,g_2^*).}
	\arrow["{(\beta^*\alpha)^*}", from=1-1, to=1-2]
	\arrow["{\mu,\mu}", from=1-2, to=1-3]
\end{tikzcd}\] Unwrapping our definitions, we need to show that the diagram
\begin{small}
\[\begin{tikzcd}[ampersand replacement=\&]
	{(h'^t \circ h)^t \circ \bar f^{st}} \&\& {h'^t \circ h^t \circ \bar f^{st}} \\
	{((h'^t \circ h)^t \circ \bar f^s)^t} \&\& {h'^t \circ (h^t \circ \bar f^s)^t} \\
	{((h'^t \circ h)^t \circ \bar f)^{st}} \&\& {h'^t \circ (h^t \circ \bar f)^{st}} \\
	{(h'^t \circ h \circ Jf)^{st}} \&\& {h'^t \circ (h \circ Jf)^{st}} \\
	{(h'^t \circ \bar f'^{st} \circ (g_1,g_2))^{st}} \&\& {h'^t \circ (\bar f'^{st} \circ (g_1,g_2))^{st}} \\
	{(\bar f''^{st} \circ (g_1'^tg_1,g_2'^tg_2))^{st}} \&\& {h'^t \circ (\bar f'^{ts} \circ (g_1,g_2))^{st}} \\
	{(\bar f''^{ts} \circ (g_1'^tg_1,g_2'^tg_2))^{st}} \&\& {h'^t \circ (\bar f'^{ts} \circ (g_1^t,g_2))^t} \\
	{(\bar f''^{ts} \circ ((g_1'^tg_1)^t,g_2'^tg_2))^t} \&\& {h'^t \circ (\bar f'^{ts} \circ (g_1^t,g_2))^t} \\
	{(\bar f''^{st} \circ ((g_1'^tg_1)^t,g_2'^tg_2))^t} \&\& {h'^t \circ \bar f'^{ts} \circ (g_1^t,g_2^t)} \\
	{ \bar f''^{st} \circ ((g_1'^tg_1)^t,(g_2'^tg_2)^t)} \& { \bar f''^{st} \circ (g_1'^tg_1^t,(g_2'^tg_2)^t)} \& {\bar f''^{ts} \circ (g_1'^t,g_2'^t) \circ (g_1^t,g_2^t)}
	\arrow["{\hat t}"', from=10-1, to=10-2]
	\arrow["{\hat t^{-1}}"', from=1-1, to=2-1]
	\arrow["{\hat s^{-1}}"', from=2-1, to=3-1]
	\arrow["{\tilde t^{-1}}"', from=3-1, to=4-1]
	\arrow["\alpha"', from=4-1, to=5-1]
	\arrow["{\beta^*}"', from=5-1, to=6-1]
	\arrow["{\gamma^{-1}}"', from=6-1, to=7-1]
	\arrow["{\hat s}"', from=7-1, to=8-1]
	\arrow["\gamma"', from=8-1, to=9-1]
	\arrow["{\hat t}"', from=9-1, to=10-1]
	\arrow["{\hat t}", from=1-1, to=1-3]
	\arrow["{\hat t^{-1}}", from=1-3, to=2-3]
	\arrow["{\hat s^{-1}}", from=2-3, to=3-3]
	\arrow["{\tilde t^{-1}}", from=3-3, to=4-3]
	\arrow["\alpha", from=4-3, to=5-3]
	\arrow["{\gamma^{-1}}", from=5-3, to=6-3]
	\arrow["{\hat s}", from=6-3, to=7-3]
	\arrow["\gamma", from=7-3, to=8-3]
	\arrow["{\hat t}", from=8-3, to=9-3]
	\arrow["{\beta^*}", from=9-3, to=10-3]
	\arrow["{\hat t}"', from=10-2, to=10-3]
\end{tikzcd}\]
\end{small} commutes. Filling this diagram is involved, but aside from naturality squares we require only
\begin{itemize}
    \item five instances of the pentagon axiom (1) from Definition~\ref{paramrelpsmonad}, and 
    \item axioms (3) and (4) from Definition~\ref{pscom}.
\end{itemize}
Thus also the second coherence condition holds.

For the third and final coherence condition, we begin with the composite 2-cell
% https://q.uiver.app/?q=WzAsMyxbMCwwLCJpXiogXFxjaXJjIFRmIl0sWzEsMCwiVGZcXGNpcmMoaSxpKSJdLFsyLDAsIlRmIl0sWzAsMSwiXFxiYXJcXGltYXRoX2ZeKiJdLFsxLDIsIlxcdGhldGEsXFx0aGV0YSJdXQ==
\[\begin{tikzcd}[ampersand replacement=\&]
	{i^* \circ Tf} \& {Tf\circ(i,i)} \& Tf
	\arrow["{\bar\imath_f^*}", from=1-1, to=1-2]
	\arrow["{\theta,\theta}", from=1-2, to=1-3]
\end{tikzcd}\] and must show that it is equal to 
% https://q.uiver.app/?q=WzAsMixbMCwwLCJpXiogXFxjaXJjIFRmIl0sWzEsMCwiVGYiXSxbMCwxLCJcXHRoZXRhIl1d
\[\begin{tikzcd}[ampersand replacement=\&]
	{i^* \circ Tf} \& Tf.
	\arrow["\theta", from=1-1, to=1-2]
\end{tikzcd}\] Rewriting everything in our terms shows that we must show the diagram
% https://q.uiver.app/?q=WzAsMTIsWzAsMCwiaV50IFxcY2lyYyBcXGJhciBmXntzdH0iXSxbMiwwLCJcXGJhciBmXntzdH0iXSxbMCwxLCIoaV50IFxcY2lyYyBcXGJhciBmXnMpXnQiXSxbMCwyLCIoaV50IFxcY2lyYyBcXGJhciBmKV57c3R9Il0sWzAsMywiKGkgXFxjaXJjIEpmKV57c3R9Il0sWzAsNCwiKFxcYmFyIGZecyBcXGNpcmNfcyBpKV57c3R9Il0sWzAsNSwiKFxcYmFyIGZee3N0fSBcXGNpcmMgKGksaSkpXntzdH0iXSxbMiw1LCIoXFxiYXIgZl57dHN9IFxcY2lyYyAoaSxpKSlee3N0fSJdLFsyLDQsIihcXGJhciBmXnt0c30gXFxjaXJjIChpXnQsaSkpXnQiXSxbMiwzLCIoXFxiYXIgZl57c3R9IFxcY2lyYyAoaV50LGkpKV50Il0sWzIsMiwiXFxiYXIgZl57c3R9IFxcY2lyYyAoaV50LGledCkiXSxbMiwxLCJcXGJhciBmXntzdH0gXFxjaXJjX3MgaV50Il0sWzAsMiwiXFxoYXQgdF57LTF9IiwyXSxbMiwzLCJcXGhhdCBzXnstMX0iLDJdLFszLDQsIlxcdGlsZGUgdF57LTF9IiwyXSxbNCw1LCJcXHRpbGRlIHMiLDJdLFs1LDYsIlxcdGlsZGUgdCIsMl0sWzYsNywiXFxnYW1tYV57LTF9IiwyXSxbNyw4LCJcXGhhdCBzIiwyXSxbOCw5LCJcXGdhbW1hIiwyXSxbOSwxMCwiXFxoYXQgdCIsMl0sWzAsMSwiXFx0aGV0YSJdLFsxMSwxLCJcXHRoZXRhIiwyXSxbMTAsMTEsIlxcdGhldGEiLDJdXQ==
\[\begin{tikzcd}[ampersand replacement=\&]
	{i^t \circ \bar f^{st}} \&\& {\bar f^{st}} \\
	{(i^t \circ \bar f^s)^t} \&\& {\bar f^{st} \circ_s i^t} \\
	{(i^t \circ \bar f)^{st}} \&\& {\bar f^{st} \circ (i^t,i^t)} \\
	{(i \circ Jf)^{st}} \&\& {(\bar f^{st} \circ (i^t,i))^t} \\
	{(\bar f^s \circ_s i)^{st}} \&\& {(\bar f^{ts} \circ (i^t,i))^t} \\
	{(\bar f^{st} \circ (i,i))^{st}} \&\& {(\bar f^{ts} \circ (i,i))^{st}}
	\arrow["{\hat t^{-1}}"', from=1-1, to=2-1]
	\arrow["{\hat s^{-1}}"', from=2-1, to=3-1]
	\arrow["{\tilde t^{-1}}"', from=3-1, to=4-1]
	\arrow["{\tilde s}"', from=4-1, to=5-1]
	\arrow["{\tilde t}"', from=5-1, to=6-1]
	\arrow["{\gamma^{-1}}"', from=6-1, to=6-3]
	\arrow["{\hat s}"', from=6-3, to=5-3]
	\arrow["\gamma"', from=5-3, to=4-3]
	\arrow["{\hat t}"', from=4-3, to=3-3]
	\arrow["\theta", from=1-1, to=1-3]
	\arrow["\theta"', from=2-3, to=1-3]
	\arrow["\theta"', from=3-3, to=2-3]
\end{tikzcd}\] commutes. Filling the diagram requires, aside from naturality squares:
\begin{itemize}
    \item instances of equalities (2) and (3) from Lemma~\ref{param3moreeqs},
    \item two uses of axiom (2) from Definition~\ref{paramrelpsmonad}, and
    \item axioms (1) and (2) from Definition~\ref{pscom}.
\end{itemize}
Hence the final coherence condition is satisfied, and thus we have shown that every pseudocommutative relative pseudomonad is a multicategorical relative pseudomonad.
\end{proof}

As the above proof demonstrates, working directly with pseudocommutativity and multicategoricality can be tedious. In the next section we will examine a condition on a relative pseudomonad which both implies pseudocommutativity and which is much easier to verify, being characterised by a universal property.

\section{Lax idempotency}

We will now consider a special class of relative pseudomonads. Defined in \cite{fghw}, the \emph{lax-idempotent} relative pseudomonad generalises the notion of a lax-idempotent or Kock-Zöberlein 2-monad, discussed extensively in \cite{kock1995}. The aim of this section is to generalise the result of López Franco in \cite{lopezfranco2011} that every lax-idempotent 2-monad is pseudocommutative.

First, we recall the definition of lax-idempotent relative pseudomonad from \cite{fghw}.

\begin{definition}
    (Lax-idempotent relative pseudomonad) Let $T : \D \to \C$ be a relative pseudomonad along $J : \D \to \C$. We say that $T$ is a \emph{lax-idempotent relative pseudomonad} if `monad structure is left adjoint to unit', which is to say that we have an adjunction
    % https://q.uiver.app/?q=WzAsMixbMCwwLCJcXEMoSlgsVFkpIl0sWzIsMCwiXFxDKFRYLFRZKSJdLFsxLDAsIi0gXFxjaXJjIGkiLDAseyJjdXJ2ZSI6LTJ9XSxbMCwxLCIoLSleKiIsMCx7ImN1cnZlIjotMn1dLFszLDIsIiIsMCx7ImxldmVsIjoxLCJzdHlsZSI6eyJuYW1lIjoiYWRqdW5jdGlvbiJ9fV1d
\[\begin{tikzcd}[ampersand replacement=\&]
	{\C(JX,TY)} \&\& {\C(TX,TY)}
	\arrow[""{name=0, anchor=center, inner sep=0}, "{- \circ i}", curve={height=-12pt}, from=1-3, to=1-1]
	\arrow[""{name=1, anchor=center, inner sep=0}, "{(-)^*}", curve={height=-12pt}, from=1-1, to=1-3]
	\arrow["\dashv"{anchor=center, rotate=-90}, draw=none, from=1, to=0]
\end{tikzcd}\] for all objects $X,Y$ of $\D$, whose unit $- \implies (-)^*i$ has components given by the $\eta_f : f \to f^*i$ from the pseudomonadic structure (note in particular that the unit is thus invertible).
\end{definition}

\begin{remark}\label{laxid via kan}
    The definition of lax idempotency is given equivalently in \cite{fghw} in terms of Kan extensions: $T$ is lax-idempotent if for all maps $f : JX \to TY$ the diagram
    % https://q.uiver.app/?q=WzAsMyxbMCwwLCJKWCJdLFsxLDAsIlRYIl0sWzEsMSwiVFkiXSxbMSwyLCJmXioiXSxbMCwyLCJmIiwyXSxbMCwxLCJpIl0sWzQsMywiXFxldGFfZiIsMCx7InNob3J0ZW4iOnsic291cmNlIjoyMCwidGFyZ2V0IjoyMH19XV0=
\[\begin{tikzcd}[ampersand replacement=\&]
	JX \& TX \\
	\& TY
	\arrow[""{name=0, anchor=center, inner sep=0}, "{f^*}", from=1-2, to=2-2]
	\arrow[""{name=1, anchor=center, inner sep=0}, "f"', from=1-1, to=2-2]
	\arrow["i", from=1-1, to=1-2]
	\arrow["{\eta_f}", shorten <=3pt, shorten >=3pt, Rightarrow, from=1, to=0]
\end{tikzcd}\] exhibits $f^*$ as the left Kan extension of $f$ along $i$. This form of the definition makes it immediate from the construction of $\Psh$ as a relative pseudomonad that $\Psh$ is lax-idempotent.
\end{remark}

We turn to showing that every lax-idempotent relative pseudomonad is pseudocommutative. Just as in Section 3 we defined the notion of strong relative pseudomonad for the multicategorical setting, we will define the notion of \emph{lax-idempotent strong relative pseudomonad} as follows: 

\begin{definition}
    (Lax-idempotent strong relative pseudomonad) Let $J : \D \to \C$ be a pseudo-multifunctor and let $T$ be a strong relative pseudomonad along $J$. We say $T$ is a \emph{lax-idempotent strong relative pseudomonad} if the strength is left adjoint to precomposition with the unit. That is, we have an adjunction
    % https://q.uiver.app/?q=WzAsMixbMCwwLCJcXEMoQl8xLC4uLixCX3tqLTF9LEpYLEJfe2orMX0sLi4uLEJfbjtUWSkiXSxbMiwwLCJcXEMoQl8xLC4uLixCX3tqLTF9LFRYLEJfe2orMX0sLi4uLEJfbjtUWSkiXSxbMCwxLCIoLSlee3Rfan0iLDAseyJjdXJ2ZSI6LTJ9XSxbMSwwLCItIFxcY2lyY19qICBpX1giLDAseyJjdXJ2ZSI6LTJ9XSxbMiwzLCIiLDAseyJsZXZlbCI6MSwic3R5bGUiOnsibmFtZSI6ImFkanVuY3Rpb24ifX1dXQ==
\[\begin{tikzcd}[ampersand replacement=\&]
	{\C(B_1,...,JX,...,B_n;TY)} \&\& {\C(B_1,...,TX,...,B_n;TY)}
	\arrow[""{name=0, anchor=center, inner sep=0}, "{(-)^{t_j}}", curve={height=-12pt}, from=1-1, to=1-3]
	\arrow[""{name=1, anchor=center, inner sep=0}, "{- \circ_j  i_X}", curve={height=-12pt}, from=1-3, to=1-1]
	\arrow["\dashv"{anchor=center, rotate=-90}, draw=none, from=0, to=1]
\end{tikzcd}\] for every $1 \leq j \leq n$ and objects $B_1,...,B_{j-1},JX,B_{j+1},...,B_n;TY$ whose unit $- \implies (-)^{t_j} \circ_j i$ has components \[\tilde t_{f} : f \to f^{t_j} \circ_j i_{X}\] obtained from the strong structure (again the unit is invertible).
\end{definition}

As in Remark~\ref{laxid via kan} above, we can equivalently state this condition in terms of left Kan extensions: $T$ is lax-idempotent strong if for every map $f : B_1,...,JX,...,B_n \to TY$ the diagram
% https://q.uiver.app/?q=WzAsMyxbMCwwLCJCXzEsLi4uLEpYLC4uLixCX24iXSxbMiwwLCJCXzEsLi4uLFRYLC4uLixCX24iXSxbMiwyLCJUWSJdLFsxLDIsImZee3Rfan0iXSxbMCwxLCIxLC4uLixpLC4uLiwxIl0sWzAsMiwiZiIsMl0sWzUsMywiXFx0aWxkZSB0X2YiLDAseyJzaG9ydGVuIjp7InNvdXJjZSI6MjAsInRhcmdldCI6MjB9fV1d
\[\begin{tikzcd}[ampersand replacement=\&]
	{B_1,...,JX,...,B_n} \&\& {B_1,...,TX,...,B_n} \\
	\\
	\&\& TY
	\arrow[""{name=0, anchor=center, inner sep=0}, "{f^{t_j}}", from=1-3, to=3-3]
	\arrow["{1,...,i,...,1}", from=1-1, to=1-3]
	\arrow[""{name=1, anchor=center, inner sep=0}, "f"', from=1-1, to=3-3]
	\arrow["{\tilde t_f}", shorten <=10pt, shorten >=10pt, Rightarrow, from=1, to=0]
\end{tikzcd}\] exhibits $f^{t_j}$ as the left Kan extension of $f$ along $1,...,i,...,1$. As a point of notation, we will use Greek letters to denote the \emph{counit} of the lax idempotency adjunction; where the strengthening map is called $(-)^t$ and the unit $\tilde t$, the counit will be called
\[\tau_f : (f \circ_t i)^t \to f,\] and where the strengthening is called $(-)^s$ and the unit $\tilde s$, the counit shall be called \[\sigma_f : (f \circ_s i)^s \to f\] (and similarly for $(-)^u$ etc.).

Note that there is much less data to check in the course of showing that a relative pseudomonad is lax-idempotent compared with showing that it is pseudocommutative. The following result generalising \cite{lopezfranco2011} therefore gives us a shortcut for showing relative pseudomonads like $\Psh$ are pseudocommutative (and hence by Theorem~\ref{pscom=>multmonad} a multicategorical relative pseudomonad).

\begin{theorem}\label{laxid=>pscom}
    Let $T : \D \to \C$ be a lax-idempotent strong relative pseudomonad. Then $T$ is pseudocommutative, with a pseudocommutativity whose components $\gamma_g : g^{ts} \to g^{st}$ are given by the composite \[g^{ts} \xrightarrow{(\tilde s_g)^{ts}} (g^s \circ_s i)^{ts} \xrightarrow{\sim} (g^{st} \circ_s i)^s \xrightarrow{\sigma_{g^{st}}} g^{st}.\]
\end{theorem}
\begin{proof}
    To begin, we first show that that putative $\gamma_g$ is invertible. We will show that the composite \[g^{st} \xrightarrow{(\tilde t_g)^{st}} (g^t \circ_t i)^{st} \xrightarrow{\sim} (g^{ts} \circ_t i)^t \xrightarrow{\tau_{g^{ts}}} g^{ts}\] is its inverse. We have the commuting diagram
    % https://q.uiver.app/?q=WzAsMTQsWzAsMCwiZ157dHN9Il0sWzEsMCwiKGdecyBcXGNpcmNfcyBpKV57dHN9Il0sWzIsMCwiKGdee3N0fSBcXGNpcmNfcyBpKV5zIl0sWzMsMCwiZ157c3R9Il0sWzMsMSwiKGdedCBcXGNpcmNfdCBpKV57c3R9Il0sWzMsMiwiKGdee3RzfSBcXGNpcmNfdCBpKV50Il0sWzMsMywiZ157dHN9Il0sWzIsMSwiKChnXnQgXFxjaXJjX3QgaSlee3N0fSBcXGNpcmNfcyBpKV5zIl0sWzEsMSwiKChnXnQgXFxjaXJjX3QgaSlecyBcXGNpcmNfcyBpKV57dHN9Il0sWzIsMiwiKChnXnt0c30gXFxjaXJjX3QgaSledCBcXGNpcmNfcyBpKV5zIl0sWzIsMywiKGdee3RzfVxcY2lyY19zIGkpXnMiXSxbMSwyLCIoKGdee3RzfSBcXGNpcmNfdCBpKSBcXGNpcmNfcyBpKV57dHN9Il0sWzAsMSwiKGdedCBcXGNpcmNfdCBpKV57dHN9Il0sWzAsMywiKChnXnt0c30gXFxjaXJjX3MgaSkgXFxjaXJjX3QgaSlee3RzfSJdLFswLDEsIlxcdGlsZGUgcyJdLFsxLDIsIlxcc2ltIl0sWzIsMywiXFxzaWdtYSJdLFszLDQsIlxcdGlsZGUgdCJdLFs0LDUsIlxcc2ltIl0sWzUsNiwiXFx0YXUiXSxbMiw3LCJcXHRpbGRlIHQiXSxbNyw0LCJcXHNpZ21hIl0sWzEsOCwiXFx0aWxkZSB0Il0sWzgsNywiXFxzaW0iXSxbNyw5LCJcXHNpbSJdLFs5LDUsIlxcc2lnbWEiXSxbOSwxMCwiXFx0YXUiXSxbMTAsNiwiXFxzaWdtYSIsMl0sWzgsMTEsIlxcc2ltIiwyXSxbMTEsOSwiXFxzaW0iLDJdLFswLDEyLCJcXHRpbGRlIHQiLDJdLFsxMiw4LCJcXHRpbGRlIHMiXSxbMTMsMTEsIiIsMix7ImxldmVsIjoyLCJzdHlsZSI6eyJoZWFkIjp7Im5hbWUiOiJub25lIn19fV0sWzEyLDEzLCJcXHRpbGRlIHMiLDJdLFsxMywxMCwiXFx0YXUiLDJdXQ==
\[\begin{tikzcd}[ampersand replacement=\&]
	{g^{ts}} \& {(g^s \circ_s i)^{ts}} \& {(g^{st} \circ_s i)^s} \& {g^{st}} \\
	{(g^t \circ_t i)^{ts}} \& {((g^t \circ_t i)^s \circ_s i)^{ts}} \& {((g^t \circ_t i)^{st} \circ_s i)^s} \& {(g^t \circ_t i)^{st}} \\
	\& {((g^{ts} \circ_t i) \circ_s i)^{ts}} \& {((g^{ts} \circ_t i)^t \circ_s i)^s} \& {(g^{ts} \circ_t i)^t} \\
	{((g^{ts} \circ_s i) \circ_t i)^{ts}} \&\& {(g^{ts}\circ_s i)^s} \& {g^{ts}}
	\arrow["{\tilde s}", from=1-1, to=1-2]
	\arrow["\sim", from=1-2, to=1-3]
	\arrow["\sigma", from=1-3, to=1-4]
	\arrow["{\tilde t}", from=1-4, to=2-4]
	\arrow["\sim", from=2-4, to=3-4]
	\arrow["\tau", from=3-4, to=4-4]
	\arrow["{\tilde t}", from=1-3, to=2-3]
	\arrow["\sigma", from=2-3, to=2-4]
	\arrow["{\tilde t}", from=1-2, to=2-2]
	\arrow["\sim", from=2-2, to=2-3]
	\arrow["\sim", from=2-3, to=3-3]
	\arrow["\sigma", from=3-3, to=3-4]
	\arrow["\tau", from=3-3, to=4-3]
	\arrow["\sigma"', from=4-3, to=4-4]
	\arrow["\sim"', from=2-2, to=3-2]
	\arrow["\sim"', from=3-2, to=3-3]
	\arrow["{\tilde t}"', from=1-1, to=2-1]
	\arrow["{\tilde s}", from=2-1, to=2-2]
	\arrow[Rightarrow, no head, from=4-1, to=3-2]
	\arrow["{\tilde s}"', from=2-1, to=4-1]
	\arrow["\tau"', from=4-1, to=4-3]
\end{tikzcd}\] whose clockwise composite is the composite $(\gamma_g)^{-1} \circ \gamma_g$, entirely composed of naturality squares. Then by the following diagram
% https://q.uiver.app/?q=WzAsNixbMCwwLCJnXnt0c30iXSxbMiwyLCJnXnt0c30iXSxbMSwyLCIoZ157dHN9XFxjaXJjX3MgaSlecyJdLFswLDEsIihnXnQgXFxjaXJjX3QgaSlee3RzfSJdLFswLDIsIigoZ157dHN9IFxcY2lyY19zIGkpIFxcY2lyY190IGkpXnt0c30iXSxbMSwxLCJnXnt0c30iXSxbMiwxLCJcXHNpZ21hIiwyXSxbMCwzLCJcXHRpbGRlIHQiLDJdLFszLDQsIlxcdGlsZGUgcyIsMl0sWzQsMiwiXFx0YXUiLDJdLFszLDUsIlxcdGF1IiwyXSxbNSwyLCJcXHRpbGRlIHMiLDJdLFswLDUsIiIsMix7ImxldmVsIjoyLCJzdHlsZSI6eyJoZWFkIjp7Im5hbWUiOiJub25lIn19fV0sWzUsMSwiIiwyLHsibGV2ZWwiOjIsInN0eWxlIjp7ImhlYWQiOnsibmFtZSI6Im5vbmUifX19XV0=
\[\begin{tikzcd}[ampersand replacement=\&]
	{g^{ts}} \\
	{(g^t \circ_t i)^{ts}} \& {g^{ts}} \\
	{((g^{ts} \circ_s i) \circ_t i)^{ts}} \& {(g^{ts}\circ_s i)^s} \& {g^{ts}}
	\arrow["\sigma"', from=3-2, to=3-3]
	\arrow["{\tilde t}"', from=1-1, to=2-1]
	\arrow["{\tilde s}"', from=2-1, to=3-1]
	\arrow["\tau"', from=3-1, to=3-2]
	\arrow["\tau"', from=2-1, to=2-2]
	\arrow["{\tilde s}"', from=2-2, to=3-2]
	\arrow[Rightarrow, no head, from=1-1, to=2-2]
	\arrow[Rightarrow, no head, from=2-2, to=3-3]
\end{tikzcd}\] composed of a naturality square and two triangle identities, the anticlockwise composite of the first diagram is equal to the identity on $g^{ts}$, as required. The same argument (swapping the roles of $s$ and $t$) demonstrates that the other composite $\gamma_g \circ (\gamma_g)^{-1}$ is also the identity, and so our $\gamma_g$ is indeed invertible.

We now must show that our $\gamma_g$ satisfies the coherence conditions for a pseudocommutativity. For the unit condition
% https://q.uiver.app/?q=WzAsNCxbMCwwLCJnXnMiXSxbMSwwLCIoZ150XFxjaXJjX3QgaSlecyJdLFsyLDAsImZee3RzfVxcY2lyY190IGkiXSxbMiwxLCJmXntzdH1cXGNpcmNfdCBpIl0sWzAsMSwiKFxcdGlsZGUgdF9nKV5zIl0sWzEsMiwiXFxzaW0iXSxbMiwzLCJcXGdhbW1hX2dcXGNpcmNfdCBpIl0sWzAsMywiXFx0aWxkZSB0X3tnXnN9IiwyXV0=
\[\begin{tikzcd}[ampersand replacement=\&]
	{g^s} \& {(g^t\circ_t i)^s} \& {f^{ts}\circ_t i} \\
	\&\& {f^{st}\circ_t i}
	\arrow["{(\tilde t_g)^s}", from=1-1, to=1-2]
	\arrow["\sim", from=1-2, to=1-3]
	\arrow["{\gamma_g\circ_t i}", from=1-3, to=2-3]
	\arrow["{\tilde t_{g^s}}"', from=1-1, to=2-3]
\end{tikzcd}\] we write out $\gamma_g \circ_t i$ in terms of our composite and construct the commuting diagram
% https://q.uiver.app/?q=WzAsMTAsWzIsMCwiZ157dHN9XFxjaXJjX3QgaSJdLFsyLDMsImdee3N0fVxcY2lyY190IGkiXSxbMSwwLCIoZ150XFxjaXJjX3QgaSlecyJdLFswLDAsImdecyJdLFsyLDEsIihnXnMgXFxjaXJjX3MgaSlee3RzfVxcY2lyY190IGkiXSxbMiwyLCIoZ157c3R9IFxcY2lyY19zIGkpXnNcXGNpcmNfdCBpIl0sWzEsMSwiKChnXnMgXFxjaXJjX3MgaSledFxcY2lyY190IGkpXnMiXSxbMCwxLCIoZ15zIFxcY2lyY19zIGkpXnMiXSxbMSwyLCIoKGdee3N0fSBcXGNpcmNfcyBpKVxcY2lyY190IGkpXnMiXSxbMCwzLCIoKGdee3N0fSBcXGNpcmNfdCBpKVxcY2lyY19zIGkpXnMiXSxbMiwwLCJcXHNpbSJdLFszLDIsIlxcdGlsZGUgdCJdLFswLDQsIlxcdGlsZGUgcyJdLFs0LDUsIlxcc2ltIl0sWzUsMSwiXFxzaWdtYSJdLFsyLDYsIlxcdGlsZGUgcyJdLFs2LDQsIlxcc2ltIl0sWzMsNywiXFx0aWxkZSBzIiwyXSxbNyw2LCJcXHRpbGRlIHQiXSxbNiw4LCJcXHNpbSJdLFs4LDUsIlxcc2ltIl0sWzgsOSwiIiwwLHsibGV2ZWwiOjIsInN0eWxlIjp7ImhlYWQiOnsibmFtZSI6Im5vbmUifX19XSxbOSwxLCJcXHNpZ21hIiwyXSxbNyw5LCJcXHRpbGRlIHQiLDJdXQ==
\[\begin{tikzcd}[ampersand replacement=\&]
	{g^s} \& {(g^t\circ_t i)^s} \& {g^{ts}\circ_t i} \\
	{(g^s \circ_s i)^s} \& {((g^s \circ_s i)^t\circ_t i)^s} \& {(g^s \circ_s i)^{ts}\circ_t i} \\
	\& {((g^{st} \circ_s i)\circ_t i)^s} \& {(g^{st} \circ_s i)^s\circ_t i} \\
	{((g^{st} \circ_t i)\circ_s i)^s} \&\& {g^{st}\circ_t i}
	\arrow["\sim", from=1-2, to=1-3]
	\arrow["{\tilde t}", from=1-1, to=1-2]
	\arrow["{\tilde s}", from=1-3, to=2-3]
	\arrow["\sim", from=2-3, to=3-3]
	\arrow["\sigma", from=3-3, to=4-3]
	\arrow["{\tilde s}", from=1-2, to=2-2]
	\arrow["\sim", from=2-2, to=2-3]
	\arrow["{\tilde s}"', from=1-1, to=2-1]
	\arrow["{\tilde t}", from=2-1, to=2-2]
	\arrow["\sim", from=2-2, to=3-2]
	\arrow["\sim", from=3-2, to=3-3]
	\arrow[Rightarrow, no head, from=3-2, to=4-1]
	\arrow["\sigma"', from=4-1, to=4-3]
	\arrow["{\tilde t}"', from=2-1, to=4-1]
\end{tikzcd}\] comprising five naturality squares. Then the anticlockwise composite is, by the following commuting diagram
% https://q.uiver.app/?q=WzAsNSxbMSwyLCJnXntzdH1cXGNpcmNfdCBpIl0sWzAsMCwiZ15zIl0sWzAsMSwiKGdecyBcXGNpcmNfcyBpKV5zIl0sWzAsMiwiKChnXntzdH0gXFxjaXJjX3QgaSlcXGNpcmNfcyBpKV5zIl0sWzEsMSwiZ15zIl0sWzEsMiwiXFx0aWxkZSBzIiwyXSxbMywwLCJcXHNpZ21hIiwyXSxbMiwzLCJcXHRpbGRlIHQiLDJdLFsyLDQsIlxcc2lnbWEiLDJdLFsxLDQsIiIsMSx7ImxldmVsIjoyLCJzdHlsZSI6eyJoZWFkIjp7Im5hbWUiOiJub25lIn19fV0sWzQsMCwiXFx0aWxkZSB0Il1d
\[\begin{tikzcd}[ampersand replacement=\&]
	{g^s} \\
	{(g^s \circ_s i)^s} \& {g^s} \\
	{((g^{st} \circ_t i)\circ_s i)^s} \& {g^{st}\circ_t i}
	\arrow["{\tilde s}"', from=1-1, to=2-1]
	\arrow["\sigma"', from=3-1, to=3-2]
	\arrow["{\tilde t}"', from=2-1, to=3-1]
	\arrow["\sigma"', from=2-1, to=2-2]
	\arrow[Rightarrow, no head, from=1-1, to=2-2]
	\arrow["{\tilde t}", from=2-2, to=3-2]
\end{tikzcd}\] of a naturality square and a triangle identity, equal to $\tilde t_{g^s}$, as required. The other unit condition is shown by the same argument, swapping the roles of $s$ and $t$.

Now, for the strengthening condition
% https://q.uiver.app/?q=WzAsNyxbMCwwLCIoZl50XFxjaXJjX3QgZylee3RzfSJdLFswLDEsIihmXnRcXGNpcmNfdCBnKV57c3R9Il0sWzAsMiwiKGZee3RzfVxcY2lyY190IGcpXnQiXSxbMywwLCJmXnt0c31cXGNpcmNfdCBnXnQiXSxbMywyLCJmXntzdH1cXGNpcmNfdCBnXnQiXSxbMiwyLCIoZl57c3R9XFxjaXJjX3QgZyledCJdLFsyLDAsIihmXnRcXGNpcmNfdCBnXnQpXnMiXSxbMCwxLCJcXGdhbW1hX3tmXnRcXGNpcmNfdCBnfSIsMl0sWzEsMiwiXFxzaW0iLDJdLFszLDQsIlxcZ2FtbWFfZlxcY2lyY190IGdedCJdLFswLDYsIihcXGhhdCB0X3tmLGd9KV5zIl0sWzYsMywiXFxzaW0iXSxbMiw1LCIoXFxnYW1tYV9mXFxjaXJjX3QgZyledCIsMl0sWzUsNCwiXFxoYXQgdF97Zl5zLGd9IiwyXV0=
\[\begin{tikzcd}[ampersand replacement=\&]
	{(f^t\circ_t g)^{ts}} \&\& {(f^t\circ_t g^t)^s} \& {f^{ts}\circ_t g^t} \\
	{(f^t\circ_t g)^{st}} \\
	{(f^{ts}\circ_t g)^t} \&\& {(f^{st}\circ_t g)^t} \& {f^{st}\circ_t g^t}
	\arrow["{\gamma_{f^t\circ_t g}}"', from=1-1, to=2-1]
	\arrow["\sim"', from=2-1, to=3-1]
	\arrow["{\gamma_f\circ_t g^t}", from=1-4, to=3-4]
	\arrow["{(\hat t_{f,g})^s}", from=1-1, to=1-3]
	\arrow["\sim", from=1-3, to=1-4]
	\arrow["{(\gamma_f\circ_t g)^t}"', from=3-1, to=3-3]
	\arrow["{\hat t_{f^s,g}}"', from=3-3, to=3-4]
\end{tikzcd}\] we can write out the anticlockwise composite in terms of our $\gamma$ and construct a large commuting diagram filled in entirely with naturality squares and one triangle identity. The other strengthening condition is shown by the same argument, swapping the roles of $s$ and $t$.

Finally, for the braiding coherence condition
% https://q.uiver.app/?q=WzAsNixbMCwwLCJmXnt1dHN9Il0sWzQsMSwiZl57c3R1fSJdLFsyLDAsImZee3R1c30iXSxbNCwwLCJmXnt0c3V9Il0sWzAsMSwiZl57dXN0fSJdLFsyLDEsImZee3N1dH0iXSxbMCwyLCIoXFxnYW1tYV9mKV5zIl0sWzIsMywiXFxnYW1tYV97Zl50fSJdLFszLDEsIihcXGdhbW1hX2YpXnUiXSxbMCw0LCJcXGdhbW1hX3tmXnV9IiwyXSxbNCw1LCIoXFxnYW1tYV9mKV50IiwyXSxbNSwxLCJcXGdhbW1hX3tmXnN9IiwyXV0=
\[\begin{tikzcd}[ampersand replacement=\&]
	{f^{uts}} \&\& {f^{tus}} \&\& {f^{tsu}} \\
	{f^{ust}} \&\& {f^{sut}} \&\& {f^{stu}}
	\arrow["{(\gamma_f)^s}", from=1-1, to=1-3]
	\arrow["{\gamma_{f^t}}", from=1-3, to=1-5]
	\arrow["{(\gamma_f)^u}", from=1-5, to=2-5]
	\arrow["{\gamma_{f^u}}"', from=1-1, to=2-1]
	\arrow["{(\gamma_f)^t}"', from=2-1, to=2-3]
	\arrow["{\gamma_{f^s}}"', from=2-3, to=2-5]
\end{tikzcd}\] after writing each composite in terms of our $\gamma$ we obtain a large diagram that may be filled in entirely with naturality squares. So all five coherence conditions are satisfied and hence indeed our $\gamma$ is a pseudocommutativity for $T$.
\end{proof}

In summary, the previous sections have proved the following implications for $T$ a relative pseudomonad along $J : \D \to \C$ between 2-multicategories:

\begin{itemize}
    \item Every strong relative pseudomonad $T$ is a pseudo-multifunctor (Proposition~\ref{param=>psfunctor}).
    \item Every pseudocommutative relative pseudomonad $T$ is a multicategorical relative pseudomonad (Theorem~\ref{pscom=>multmonad}).
    \item Every lax-idempotent strong relative pseudomonad $T$ is pseudocommutative (Theorem~\ref{laxid=>pscom}).
\end{itemize}

\section{The presheaf relative pseudomonad}

We apply our results to the presheaf construction. As shown in \cite{fghw}, the presheaf construction $\Psh : \Cat \to \CAT : X \mapsto \Psh X := [X^{op},\Set]$ can be given the structure of a relative pseudomonad, where the units are given by the Yoneda embedding $y_X : X \to \Psh X$ and the extension of a functor $f : X \to \Psh Y$ for small categories $X,Y$ is given by the left Kan extension
% https://q.uiver.app/?q=WzAsMyxbMCwwLCJYIl0sWzEsMCwiXFxQc2ggWCJdLFsxLDEsIlxcUHNoIFkiXSxbMSwyLCJmXiogOj0gXFxMYW5feSBmIl0sWzAsMSwieSJdLFswLDIsImYiLDJdLFs1LDMsIlxcZXRhX2YiLDAseyJzaG9ydGVuIjp7InNvdXJjZSI6MjAsInRhcmdldCI6MjB9fV1d
\[\begin{tikzcd}[ampersand replacement=\&]
	X \& {\Psh X} \\
	\& {\Psh Y}
	\arrow[""{name=0, anchor=center, inner sep=0}, "{f^* := \Lan_y f}", from=1-2, to=2-2]
	\arrow["y", from=1-1, to=1-2]
	\arrow[""{name=1, anchor=center, inner sep=0}, "f"', from=1-1, to=2-2]
	\arrow["{\eta_f}", shorten <=4pt, shorten >=4pt, Rightarrow, from=1, to=0]
\end{tikzcd}\] along the Yoneda embedding, and this diagram also defines the map $\eta_f : f \to f^*i$.

In order to make use of the our results, we need to further show that the presheaf relative pseudomonad is strong.

\begin{proposition}\label{pshisparam}
    The presheaf relative pseudomonad $\Psh$ along $J : \Cat \to \CAT$ is strong, with the strengthening of a functor \[f : B_1,...,B_{j-1},JX,B_{j+1},...,B_n \to \Psh Y\] defined as the left Kan extension 
    % https://q.uiver.app/?q=WzAsMyxbMCwwLCJCXzEsLi4uLFgsLi4uLEJfbiJdLFsyLDAsIkJfMSwuLi4sXFxQc2ggWCwuLi4sQl9uIl0sWzIsMiwiXFxQc2ggWSJdLFsxLDIsImZeKiA6PSBcXExhbl97MSwuLi4seSwuLi4sMX0gZiJdLFswLDEsIjEsLi4uLHksLi4uLDEiXSxbMCwyLCJmIiwyXSxbNSwzLCJcXHRpbGRlIHRfZiIsMCx7InNob3J0ZW4iOnsic291cmNlIjoyMCwidGFyZ2V0IjoyMH19XV0=
\[\begin{tikzcd}[ampersand replacement=\&]
	{B_1,...,X,...,B_n} \&\& {B_1,...,\Psh X,...,B_n} \\
	\\
	\&\& {\Psh Y}
	\arrow[""{name=0, anchor=center, inner sep=0}, "{f^t := \Lan_{1,...,y,...,1} f}", from=1-3, to=3-3]
	\arrow["{1,...,y,...,1}", from=1-1, to=1-3]
	\arrow[""{name=1, anchor=center, inner sep=0}, "f"', from=1-1, to=3-3]
	\arrow["{\tilde t_f}", shorten <=11pt, shorten >=11pt, Rightarrow, from=1, to=0]
\end{tikzcd}\] along $1,...,y,...,1$, and the 2-cell in the above diagram defines the map $\tilde t_{f}$.
\end{proposition}
\begin{proof}
    We begin by constructing the rest of the data for a strong relative pseudomonad; namely, the invertible families of 2-cells
    \[\hat t_{f,g} : (f^t \circ_t g)^t \to f^t \circ g^t,\ \theta : i^t \to 1.\]
    Using the universal property of the left Kan extension, we define $\hat t_{f,g}$ and $\theta$ to be the unique 2-cells such that
    % https://q.uiver.app/?q=WzAsNixbMSwwLCIoZl50IFxcY2lyY190IGcpXnQgXFxjaXJjX3QgeSJdLFsxLDEsIihmXnQgXFxjaXJjX3QgZyledCBcXGNpcmNfdCB5Il0sWzAsMCwiZl50IFxcY2lyY190IGciXSxbMiwwLCJ5Il0sWzMsMCwieV50IFxcY2lyYyB5Il0sWzMsMSwieSJdLFswLDEsIlxcaGF0IHRfe2YsZ30gXFxjaXJjX3QgeSJdLFsyLDAsIlxcdGlsZGUgdF97Zl50IFxcY2lyY190IGd9Il0sWzIsMSwiZl50IFxcY2lyY190XFx0aWxkZSB0X2ciLDJdLFs0LDUsIlxcdGhldGEgXFxjaXJjIHkiXSxbMyw0LCJcXHRpbGRlIHRfeSJdLFszLDUsIiIsMix7ImxldmVsIjoyLCJzdHlsZSI6eyJoZWFkIjp7Im5hbWUiOiJub25lIn19fV1d
\[\begin{tikzcd}[ampersand replacement=\&]
	{f^t \circ_t g} \& {(f^t \circ_t g)^t \circ_t y} \& y \& {y^t \circ y} \\
	\& {(f^t \circ_t g)^t \circ_t y} \&\& y
	\arrow["{\hat t_{f,g} \circ_t y}", from=1-2, to=2-2]
	\arrow["{\tilde t_{f^t \circ_t g}}", from=1-1, to=1-2]
	\arrow["{f^t \circ_t\tilde t_g}"', from=1-1, to=2-2]
	\arrow["{\theta \circ y}", from=1-4, to=2-4]
	\arrow["{\tilde t_y}", from=1-3, to=1-4]
	\arrow[Rightarrow, no head, from=1-3, to=2-4]
    \end{tikzcd}\] commute, respectively. It remains to check the two coherence conditions of Definition~\ref{paramrelpsmonad}. For the first:
    % https://q.uiver.app/?q=WzAsNSxbMCwwLCIoKGZedCBcXGNpcmNfdCBnKV50IFxcY2lyY190IGgpXnQiXSxbMCwxLCIoZl50IFxcY2lyY190IGdedCBcXGNpcmNfdCBoKV50Il0sWzEsMSwiZl50IFxcY2lyY190IChnXnQgXFxjaXJjX3QgaCledCJdLFsyLDAsIihmXnQgXFxjaXJjX3QgZyledCBcXGNpcmNfdCBoXnQiXSxbMiwxLCIoZl50IFxcY2lyY190IGdedCkgXFxjaXJjX3QgaF50Il0sWzAsMSwiKFxcaGF0IHRfe2YsZ30gXFxjaXJjX3QgaCledCIsMl0sWzAsMywiXFxoYXQgdF97Zl50IFxcY2lyY190IGcsaH0iXSxbMyw0LCJcXGhhdCB0X3tmLGd9IFxcY2lyY190IGhedCJdLFsxLDIsIlxcaGF0IHRfe2YsZ150IFxcY2lyY190IGh9IiwyXSxbMiw0LCJmXnQgXFxjaXJjX3QgXFxoYXQgdF97ZyxofSIsMl1d
\[\begin{tikzcd}[ampersand replacement=\&]
	{((f^t \circ_t g)^t \circ_t h)^t} \&\& {(f^t \circ_t g)^t \circ_t h^t} \\
	{(f^t \circ_t g^t \circ_t h)^t} \& {f^t \circ_t (g^t \circ_t h)^t} \& {(f^t \circ_t g^t) \circ_t h^t}
	\arrow["{(\hat t_{f,g} \circ_t h)^t}"', from=1-1, to=2-1]
	\arrow["{\hat t_{f^t \circ_t g,h}}", from=1-1, to=1-3]
	\arrow["{\hat t_{f,g} \circ_t h^t}", from=1-3, to=2-3]
	\arrow["{\hat t_{f,g^t \circ_t h}}"', from=2-1, to=2-2]
	\arrow["{f^t \circ_t \hat t_{g,h}}"', from=2-2, to=2-3]
\end{tikzcd}\]  by the universal property of the left Kan extension it suffices to show that the diagram
\begin{small}
    % https://q.uiver.app/?q=WzAsNSxbMCwwLCIoKGZedCBcXGNpcmNfdCBnKV50IFxcY2lyY190IGgpXnQgXFxjaXJjX3QgeSJdLFswLDEsIihmXnQgXFxjaXJjX3QgZ150IFxcY2lyY190IGgpXnRcXGNpcmNfdCB5Il0sWzIsMSwiZl50IFxcY2lyY190IChnXnQgXFxjaXJjX3QgaCledFxcY2lyY190IHkiXSxbNCwwLCIoZl50IFxcY2lyY190IGcpXnQgXFxjaXJjX3QgaF50XFxjaXJjX3QgeSJdLFs0LDEsIihmXnQgXFxjaXJjX3QgZ150KSBcXGNpcmNfdCBoXnRcXGNpcmNfdCB5Il0sWzAsMSwiKFxcaGF0IHRfe2YsZ30gXFxjaXJjX3QgaCledFxcY2lyY190IHkiLDJdLFswLDMsIlxcaGF0IHRfe2ZedCBcXGNpcmNfdCBnLGh9XFxjaXJjX3QgeSJdLFszLDQsIlxcaGF0IHRfe2YsZ30gXFxjaXJjX3QgaF50XFxjaXJjX3QgeSJdLFsxLDIsIlxcaGF0IHRfe2YsZ150IFxcY2lyY190IGh9XFxjaXJjX3QgeSIsMl0sWzIsNCwiZl50IFxcY2lyY190IFxcaGF0IHRfe2csaH1cXGNpcmNfdCB5IiwyXV0=
\[\begin{tikzcd}[ampersand replacement=\&]
	{((f^t \circ_t g)^t \circ_t h)^t \circ_t y} \&\&\&\& {(f^t \circ_t g)^t \circ_t h^t\circ_t y} \\
	{(f^t \circ_t g^t \circ_t h)^t\circ_t y} \&\& {f^t \circ_t (g^t \circ_t h)^t\circ_t y} \&\& {(f^t \circ_t g^t) \circ_t h^t\circ_t y}
	\arrow["{(\hat t_{f,g} \circ_t h)^t\circ_t y}"', from=1-1, to=2-1]
	\arrow["{\hat t_{f^t \circ_t g,h}\circ_t y}", from=1-1, to=1-5]
	\arrow["{\hat t_{f,g} \circ_t h^t\circ_t y}", from=1-5, to=2-5]
	\arrow["{\hat t_{f,g^t \circ_t h}\circ_t y}"', from=2-1, to=2-3]
	\arrow["{f^t \circ_t \hat t_{g,h}\circ_t y}"', from=2-3, to=2-5]
\end{tikzcd}\]
\end{small} commutes. Rewriting terms we obtain the diagram
% https://q.uiver.app/?q=WzAsOCxbMCwwLCIoKGZedCBcXGNpcmNfdCBnKV50IFxcY2lyY190IGgpXnQgXFxjaXJjX3QgeSJdLFswLDEsIihmXnQgXFxjaXJjX3QgZ150IFxcY2lyY190IGgpXnRcXGNpcmNfdCB5Il0sWzEsMiwiZl50IFxcY2lyY190IChnXnQgXFxjaXJjX3QgaCledFxcY2lyY190IHkiXSxbMiwwLCIoZl50IFxcY2lyY190IGcpXnQgXFxjaXJjX3QgaF50XFxjaXJjX3QgeSJdLFsyLDEsIihmXnQgXFxjaXJjX3QgZ150KSBcXGNpcmNfdCBoXnRcXGNpcmNfdCB5Il0sWzEsMCwiKGZedCBcXGNpcmNfdCBnKV50IFxcY2lyY190IGgiXSxbMCwyLCJmXnQgXFxjaXJjX3QgZ150IFxcY2lyY190IGgiXSxbMiwyLCJmXnQgXFxjaXJjX3QgZ150IFxcY2lyY190IGgiXSxbMCwxLCJcXGhhdCB0IiwyXSxbMyw0LCJcXGhhdCB0Il0sWzAsNSwiXFx0aWxkZSB0XnstMX0iXSxbNSwzLCJcXHRpbGRlIHQiXSxbMSw2LCJcXHRpbGRlIHReey0xfSIsMl0sWzYsMiwiXFx0aWxkZSB0IiwyXSxbMiw3LCJcXHRpbGRlIHReey0xfSIsMl0sWzcsNCwiXFx0aWxkZSB0IiwyXV0=
\[\begin{tikzcd}[ampersand replacement=\&]
	{((f^t \circ_t g)^t \circ_t h)^t \circ_t y} \& {(f^t \circ_t g)^t \circ_t h} \& {(f^t \circ_t g)^t \circ_t h^t\circ_t y} \\
	{(f^t \circ_t g^t \circ_t h)^t\circ_t y} \&\& {(f^t \circ_t g^t) \circ_t h^t\circ_t y} \\
	{f^t \circ_t g^t \circ_t h} \& {f^t \circ_t (g^t \circ_t h)^t\circ_t y} \& {f^t \circ_t g^t \circ_t h}
	\arrow["{\hat t}"', from=1-1, to=2-1]
	\arrow["{\hat t}", from=1-3, to=2-3]
	\arrow["{\tilde t^{-1}}", from=1-1, to=1-2]
	\arrow["{\tilde t}", from=1-2, to=1-3]
	\arrow["{\tilde t^{-1}}"', from=2-1, to=3-1]
	\arrow["{\tilde t}"', from=3-1, to=3-2]
	\arrow["{\tilde t^{-1}}"', from=3-2, to=3-3]
	\arrow["{\tilde t}"', from=3-3, to=2-3]
\end{tikzcd}\] which we can fill in
% https://q.uiver.app/?q=WzAsNyxbMCwwLCIoKGZedCBcXGNpcmNfdCBnKV50IFxcY2lyY190IGgpXnQgXFxjaXJjX3QgeSJdLFswLDEsIihmXnQgXFxjaXJjX3QgZ150IFxcY2lyY190IGgpXnRcXGNpcmNfdCB5Il0sWzEsMiwiZl50IFxcY2lyY190IChnXnQgXFxjaXJjX3QgaCledFxcY2lyY190IHkiXSxbMiwwLCIoZl50IFxcY2lyY190IGcpXnQgXFxjaXJjX3QgaF50XFxjaXJjX3QgeSJdLFsyLDEsIihmXnQgXFxjaXJjX3QgZ150KSBcXGNpcmNfdCBoXnRcXGNpcmNfdCB5Il0sWzEsMCwiKGZedCBcXGNpcmNfdCBnKV50IFxcY2lyY190IGgiXSxbMSwxLCJmXnQgXFxjaXJjX3QgZ150IFxcY2lyY190IGgiXSxbMCwxLCJcXGhhdCB0IiwyXSxbMyw0LCJcXGhhdCB0Il0sWzAsNSwiXFx0aWxkZSB0XnstMX0iXSxbMSw2LCJcXHRpbGRlIHReey0xfSIsMl0sWzYsMiwiXFx0aWxkZSB0IiwyLHsiY3VydmUiOjN9XSxbNSw2LCJcXGhhdCB0IiwyXSxbMiw2LCJcXHRpbGRlIHReey0xfSIsMix7ImN1cnZlIjozfV0sWzUsMywiXFx0aWxkZSB0Il0sWzYsNCwiXFx0aWxkZSB0IiwyXV0=
\[\begin{tikzcd}[ampersand replacement=\&]
	{((f^t \circ_t g)^t \circ_t h)^t \circ_t y} \& {(f^t \circ_t g)^t \circ_t h} \& {(f^t \circ_t g)^t \circ_t h^t\circ_t y} \\
	{(f^t \circ_t g^t \circ_t h)^t\circ_t y} \& {f^t \circ_t g^t \circ_t h} \& {(f^t \circ_t g^t) \circ_t h^t\circ_t y} \\
	\& {f^t \circ_t (g^t \circ_t h)^t\circ_t y}
	\arrow["{\hat t}"', from=1-1, to=2-1]
	\arrow["{\hat t}", from=1-3, to=2-3]
	\arrow["{\tilde t^{-1}}", from=1-1, to=1-2]
	\arrow["{\tilde t^{-1}}"', from=2-1, to=2-2]
	\arrow["{\tilde t}"', curve={height=18pt}, from=2-2, to=3-2]
	\arrow["{\hat t}"', from=1-2, to=2-2]
	\arrow["{\tilde t^{-1}}"', curve={height=18pt}, from=3-2, to=2-2]
	\arrow["{\tilde t}", from=1-2, to=1-3]
	\arrow["{\tilde t}"', from=2-2, to=2-3]
\end{tikzcd}\] with two naturality squares. For the second:
% https://q.uiver.app/?q=WzAsNCxbMCwwLCJmXnQiXSxbMSwwLCIoZl50IFxcY2lyY190IHkpXnQiXSxbMiwwLCJmXnQgXFxjaXJjX3QgeV50Il0sWzIsMSwiZl50Il0sWzAsMSwiKFxcdGlsZGUgdF9mKV50Il0sWzEsMiwiXFxoYXQgdF97Zix5fSJdLFsyLDMsImZedCBcXGNpcmNfdCBcXHRoZXRhIl0sWzAsMywiIiwyLHsibGV2ZWwiOjIsInN0eWxlIjp7ImhlYWQiOnsibmFtZSI6Im5vbmUifX19XV0=
\[\begin{tikzcd}[ampersand replacement=\&]
	{f^t} \& {(f^t \circ_t y)^t} \& {f^t \circ_t y^t} \\
	\&\& {f^t}
	\arrow["{(\tilde t_f)^t}", from=1-1, to=1-2]
	\arrow["{\hat t_{f,y}}", from=1-2, to=1-3]
	\arrow["{f^t \circ_t \theta}", from=1-3, to=2-3]
	\arrow[Rightarrow, no head, from=1-1, to=2-3]
\end{tikzcd}\] again by the universal property of the left Kan extension we can equivalently show the diagram
% https://q.uiver.app/?q=WzAsNCxbMCwwLCJmXnQgXFxjaXJjX3QgeSJdLFsyLDAsIihmXnQgXFxjaXJjX3QgeSledFxcY2lyY190IHkiXSxbNCwwLCJmXnQgXFxjaXJjX3QgeV50XFxjaXJjX3QgeSJdLFs0LDEsImZedFxcY2lyY190IHkiXSxbMCwxLCIoXFx0aWxkZSB0X2YpXnRcXGNpcmNfdCB5Il0sWzEsMiwiXFxoYXQgdF97Zix5fVxcY2lyY190IHkiXSxbMiwzLCJmXnQgXFxjaXJjX3QgXFx0aGV0YVxcY2lyY190IHkiXSxbMCwzLCIiLDIseyJsZXZlbCI6Miwic3R5bGUiOnsiaGVhZCI6eyJuYW1lIjoibm9uZSJ9fX1dXQ==
\[\begin{tikzcd}[ampersand replacement=\&]
	{f^t \circ_t y} \&\& {(f^t \circ_t y)^t\circ_t y} \&\& {f^t \circ_t y^t\circ_t y} \\
	\&\&\&\& {f^t\circ_t y}
	\arrow["{(\tilde t_f)^t\circ_t y}", from=1-1, to=1-3]
	\arrow["{\hat t_{f,y}\circ_t y}", from=1-3, to=1-5]
	\arrow["{f^t \circ_t \theta\circ_t y}", from=1-5, to=2-5]
	\arrow[Rightarrow, no head, from=1-1, to=2-5]
\end{tikzcd}\] commutes. Rewriting terms we obtain
% https://q.uiver.app/?q=WzAsNSxbMCwwLCJmXnQgXFxjaXJjX3QgeSJdLFsyLDAsIihmXnQgXFxjaXJjX3QgeSledFxcY2lyY190IHkiXSxbNCwwLCJmXnQgXFxjaXJjX3QgeV50XFxjaXJjX3QgeSJdLFs0LDEsImZedFxcY2lyY190IHkiXSxbMywwLCJmXnQgXFxjaXJjX3QgeSJdLFswLDMsIiIsMix7ImxldmVsIjoyLCJzdHlsZSI6eyJoZWFkIjp7Im5hbWUiOiJub25lIn19fV0sWzAsMSwiXFx0aWxkZSB0Il0sWzEsNCwiXFx0aWxkZSB0XnstMX0iXSxbNCwyLCJcXHRpbGRlIHQiXSxbMiwzLCJcXHRpbGRlIHReey0xfSJdXQ==
\[\begin{tikzcd}[ampersand replacement=\&]
	{f^t \circ_t y} \&\& {(f^t \circ_t y)^t\circ_t y} \& {f^t \circ_t y} \& {f^t \circ_t y^t\circ_t y} \\
	\&\&\&\& {f^t\circ_t y}
	\arrow[Rightarrow, no head, from=1-1, to=2-5]
	\arrow["{\tilde t}", from=1-1, to=1-3]
	\arrow["{\tilde t^{-1}}", from=1-3, to=1-4]
	\arrow["{\tilde t}", from=1-4, to=1-5]
	\arrow["{\tilde t^{-1}}", from=1-5, to=2-5]
\end{tikzcd}\] which immediately commutes. Hence indeed $\Psh$ is as constructed a strong relative pseudomonad.
\end{proof}

We can now apply the results of this paper to the presheaf relative pseudomonad.

\begin{theorem}\label{psh is all the things}
    The presheaf relative pseudomonad is:
    \begin{enumerate}
        \item[(1)] a lax-idempotent strong relative pseudomonad,
        \item[(2)] a pseudocommutative relative pseudomonad, and
        \item[(3)] a multicategorical relative pseudomonad.
    \end{enumerate}
\end{theorem}
\begin{proof}
    By Theorem~\ref{laxid=>pscom} we know $(1) \implies (2)$, and by Theorem~\ref{pscom=>multmonad} we know $(2) \implies (3)$. So it suffices to check that $\Psh$ is lax-idempotent strong. By Proposition~\ref{pshisparam} $\Psh$ is strong, and we have diagrams
    \[\begin{tikzcd}[ampersand replacement=\&]
	{B_1,...,X,...,B_n} \&\& {B_1,...,\Psh X,...,B_n} \\
	\\
	\&\& {\Psh Y}
	\arrow[""{name=0, anchor=center, inner sep=0}, "{f^t := \Lan_{1,...,y,...,1} f}", from=1-3, to=3-3]
	\arrow["{1,...,y,...,1}", from=1-1, to=1-3]
	\arrow[""{name=1, anchor=center, inner sep=0}, "f"', from=1-1, to=3-3]
	\arrow["{\tilde t_f}", shorten <=11pt, shorten >=11pt, Rightarrow, from=1, to=0]
\end{tikzcd}\] exhibiting $f^t$ as the left Kan extension of $f$ along $1,...,y,...,1$. But this means precisely that we have an adjunction \[(-)^{t} \dashv - \circ_t y\] whose unit is $\tilde t$, as required. So indeed $\Psh$ is lax-idempotent strong, and hence also pseudocommutative and a multicategorical relative pseudomonad.
\end{proof}

\begin{small}
\noindent
    \textbf{Acknowledgements.}
    The author gives thanks for the support of the Engineering and Physical Sciences Research Council, which has funded the author's position as a post-graduate researcher at the University of Leeds. Personal thanks are given to Nicola Gambino for regular invaluable discussions, as well as to Nathanael Arkor for useful conversations.
\end{small}

\bibliographystyle{alpha}
\bibliography{main} %Prints bibliography

\end{document}